\documentclass{amsart}
\usepackage{amsmath}
\usepackage{amssymb}
\usepackage{amsthm}
\begin{document}

\newtheorem{thm}{Theorem} 
\newtheorem{conj}[thm]{Conjecture}
\newtheorem{prop}{Proposition}[section] 
\newtheorem{deff}[prop]{Definition}
\newtheorem{lem}[prop]{Lemma} 
\newtheorem{sheep}[prop]{Corollary} 
\newtheorem{fact}[prop]{Fact}
\newtheorem{example}[prop]{Example}
\newtheorem{notass}[prop]{Notation and Assumptions}
\newtheorem{quest}[prop]{Quest}

\newcommand{\sthat}{\hspace{.1cm}| \hspace{.1cm}}
\newcommand{\p}{\mathbb{P}^1}
\newcommand{\ga}{\mathbb{G}_a}
\newcommand{\gm}{\mathbb{G}_m}
\newcommand{\spec}{ \mathit{Spec} }
\newcommand{\id}{\operatorname{id} }
\newcommand{\acl}{\operatorname{acl}}
\newcommand{\aclsig}{\operatorname{acl_\sigma}}

\newcommand{\qq}{\mathbb{Q}}
\newcommand{\nn}{\mathbb{N}}
\newcommand{\zz}{\mathbb{Z}}
\newcommand{\zpz}{\zz / p\zz}
\newcommand{\ccc}{\mathbb{C}}

\newcommand{\AAA}{\mathcal{A}}
\newcommand{\BB}{\mathcal{B}}
\newcommand{\CC}{\mathcal{C}}
\newcommand{\DD}{\mathcal{D}}
\newcommand{\EE}{\mathcal{E}}
\newcommand{\FF}{\mathcal{F}}
\newcommand{\MM}{\mathcal{M}}

\newcommand{\fix}{\operatorname{fix}}
\newcommand{\qacfa}{$\mathbb{Q}$acfa}
\newcommand{\tsig}{T_\sigma}
\newcommand{\lsig}{L_\sigma}
\newcommand{\sigdeg}{\deg_\sigma}
\newcommand{\dM}{\operatorname{dM}}
\newcommand{\RM}{\operatorname{RM}}
\newcommand{\zdense}{Zariski-dense }

\newcommand{\tpsig}{\operatorname{tp_\sigma}}
\newcommand{\tptau}{\operatorname{tp_\tau}}

\def\Ind#1#2{#1\setbox0=\hbox{$#1x$}\kern\wd0\hbox to 0pt{\hss$#1\mid$\hss}
\lower.9\ht0\hbox to 0pt{\hss$#1\smile$\hss}\kern\wd0}
\def\inde{\mathop{\mathpalette\Ind{}}}
\def\Notind#1#2{#1\setbox0=\hbox{$#1x$}\kern\wd0\hbox to 0pt{\mathchardef
\nn=12854\hss$#1\nn$\kern1.4\wd0\hss}\hbox to
0pt{\hss$#1\mid$\hss}\lower.9\ht0 \hbox to
0pt{\hss$#1\smile$\hss}\kern\wd0}
\def\nind{\mathop{\mathpalette\Notind{}}}

\newcommand{\prodplus}{{\times +}}
\newcommand{\uniplus}{{\cup +}}
\newcommand{\ggll}{\operatorname{GL}}

\author{Alice Medvedev}
\thanks{The author was supported by NSF DMS-0854998 and NSF DMS-1500976.}

\title{$\mathbb{Q}$ACFA}

\begin{abstract}
We show that many nice properties of a theory $T$ follow from the corresponding properties of its reducts to finite subsignatures. If $\{ T_i \}_{i \in I}$ is a directed family of conservative expansions of first-order theories and each $T_i$ is stable (respectively, simple, rosy, dependent, submodel complete, model complete, companionable), then so is the union $T := \cup_i T_i$. In most cases, (thorn)-forking in $T$ is equivalent to (thorn)-forking of algebraic closures in some $T_i$.

This applies to fields with an action by $(\mathbb{Q}, +)$, whose reducts to finite subsignatures are interdefinable with the theory of fields with one automorphism. We show that the model companion $\mathbb{Q}$ACFA of this theory is strictly simple and has the same level of quantifier elimination and the same algebraic characterization of algebraic closure and forking independence as ACFA. The lattice of the fixed fields of the named automorphisms breaks supersimplicity in $\mathbb{Q}$ACFA, but away from these we find many (weakly) minimal formulas.
\end{abstract}

\maketitle

\section{Introduction}


The main subject of this paper is fields $K$ with a $(\mathbb{Q}, +)$-action, that is, an embedding of $(\mathbb{Q}, +)$ into $\operatorname{Aut}(K)$.

 The model theory of fields with one automorphism, also known as fields with a $(\mathbb{Z}, +)$-action, is worked out in great depth in \cite{udizoe} and \cite{udikobizoe}, where a model-companion ACFA is described and proved to be supersimple. Actions by some other groups have also been considered. The beginning of \cite{udizoe} applies just as well to fields with an action by a free group on several generators, showing that this theory admits a model-companion, and that this model-companion is simple. The theory of fields with an action by a free abelian group on several generators does not admit a model companion. In this paper we show that fields with a $(\mathbb{Q}, +)$-action are far more tame.

\begin{thm} \label{firstqacfathm}
The theory of fields with a $(\mathbb{Q}, +)$-action has a model-companion $\mathbb{Q}$ACFA. Completions of $\mathbb{Q}$ACFA, given by specifying the characteristic, are quantifier-free stable, simple, but not supersimple. In models of $\mathbb{Q}$ACFA, model-theoretic algebraic closure of a set $A$ is given by field-theoretic algebraic closure of the substructure generated by $A$. For model-theoretically algebraically closed subsets of models of $\mathbb{Q}$ACFA, forking independence is equivalent to algebraic independence.
\end{thm}

We obtain these results from the corresponding ones for ACFA by observing that every finitely-generated subgroup of $(\mathbb{Q}, +)$ is isomorphic to $(\mathbb{Z}, +)$, so the reduct of $\mathbb{Q}$ACFA to any finite subsignature is essentially a definitional expansion of ACFA.
We then show that many properties of a first-order theory follow from the corresponding properties of enough reducts of the theory. While this general idea has certainly been considered by other people, we do not believe these results are written down anywhere.

\begin{deff} A collection $\{ L_w \subset L_W \sthat w \in W \}$ of subsignatures of $L_W$ is \emph{sufficient} if every finite subsignature of $L_W$ is contained in some $L_w$. If, in addition, for all $u,v \in W$ there is some $w \in W$ such that $L_u \cup L_v \subset L_w$, we say that the collection is \emph{directed}.
\end{deff}

A sufficient collection of finite subsignatures is automatically directed. Any tail (i.e. $\{ L_w \sthat L_w \supset L_v \}$ for some $v \in W$) of a directed sufficient collection of subsignatures is itself sufficient and directed. The natural signature $L_{\mathbb{Q}} := \{ +, \cdot, 0, 1 \} \cup \{ \sigma_q \sthat q \in \mathbb{Q} \}$ for fields with a  $(\mathbb{Q}, +)$-action admits a directed sufficient collection of subsignatures $L_q := \{ +, \cdot, 0, 1 \} \cup \{ \sigma_{nq} \sthat n \in \mathbb{Z} \}$ for $q \in \qq$.


\begin{thm} \label{onepurthm}
 Suppose that $T_W$ is an $L_W$-theory, and $\{ L_w \subset L_W \sthat w \in W \}$ is a sufficient collection of subsignatures of $L_W$. For each of the following properties, if all reducts $T_w$ of $T_W$ to $L_w$ have the property, then $T_W$ does too: consistency, completeness, quantifier elimination, partial quantifier elimination such a model-completeness, elimination of imaginaries, stable embeddedness of some definable set, characterization of algebraic closure, stability, simplicity, rosiness, dependence. \end{thm}

Each of stability, simplicity, and rosiness is characterized by the presence of a good notion of (forking or thron-forking) independence; often, independence in the sense of $T_W$ can be characterized in terms of independence in some or all of the reducts $T_w$.

\begin{thm} \label{twopurthm}
Suppose that $T_W$ is an $L_W$-theory; that $\{ L_w \subset L_W \sthat w \in W \}$ is a sufficient collection of subsignatures of $L_W$; and that all $T_w$ and, therefore, $T_W$ are simple (resp. rosy). Let $A \subset B, C \subset M_W \models T_W$, and suppose that $A$, $B$, and $C$ are algebraically closed.

If $B$ is forking (resp. thorn-forking) independent from $C$ over $A$ in reducts $M_w$ of $M_W$ to $L_w$ for all $w$, then this also holds in $M_W$.

The converse is true when all $T_w$ are simple and eliminate hyperimaginaries.
\end{thm}

For stable theories, this follows immediately from the characterization of non-forking in terms of definitional extensions of types. The requirement that the three sets be algebraically closed in the sense of the full signature rules out obvious counterexamples. We do not see how to remove the requirement that all $T_w$ eliminate hyperimaginaries, though it is conjectured that all simple theories do. This result suffices for the application to $\mathbb{Q}$ACFA as the reducts are definitional expansions of ACFA, which is supersimple, and all supersimple theories eliminate hyperimaginaries. In Section \ref{combostabsec}, we point out the difficulties in trying to prove the converse for rosy theories. The one direction already has strong consequences for ranks and other notions from geometric stability theory (see Section \ref{puresect} for definitions). It follows immediately that the limsup of Lascar ranks in $L_w$, if finite, is an upper bound on the Lascar rank in $L_W$. Here, we use the term ``Lascar rank'' loosely, allowing unstable theories and partial types. Furthermore,

\begin{thm} \label{geopurthm}
Suppose that $T_W$ is an $L_W$-theory, and $\{ L_w \subset L_W \sthat w \in W \}$ is a (directed) sufficient collection of subsignatures of $L_W$.
 A partial $L_W$-type $\pi$ is $L_W$-trivial (resp., $L_W$-one-based, $L_W$-modular group) whenever all reducts $\pi_w$ are $L_w$-trivial (resp., $L_w$-one-based, $L_w$-modular group whenever the group law is $L_w$-definable).\end{thm}


 We obtain stronger results for $\mathbb{Q}$ACFA, where the various reducts are so closely related that one reduct might already control everything. While some of these rely on special properties of ACF, others are equally true when ACF is replaced by an arbitrary theory $T$, that is, for the model-companion of the theory of models of $T$ with a $\mathbb{Q}$-action, which will exist whenever $T_A$ exists. For example, it may be interesting to see how much of this works when ACF is replaced with DCF.

\begin{thm} \label{better-qacfa}
($\mathbb{Q}$ACFA) Suppose that $p$ is an $L_\mathbb{Q}$-type whose reduct $p_1$ to $L_1$ is minimal in the sense of $L_1$ and ACFA. \begin{itemize}
 \item If $p_1$ is trivial in the sense of $L_1$ and ACFA, then $p$ is minimal and trivial in the sense of $\mathbb{Q}$ACFA.
 \item If $p_1$ is fieldlike in the sense of $L_1$ and ACFA and nonorthogonal to $\sigma_1(x) = x$, then the Lascar rank of $p$ in the sense of $\mathbb{Q}$ACFA is infinity. If $p_1$ is fieldlike in the sense of $L_1$ and ACFA and nonorthogonal to the fixed field of $\sigma_1^{m'}(x) = x^{p^m}$, for some $m, m' \neq 0$, then the Lascar rank of $p$ in the sense of $\mathbb{Q}$ACFA is $m$.
 \item If $p_1$ is orthogonal to all $L_1$-definable minimal fields in the sense of $L_1$ and ACFA, then $p$ is orthogonal to all $L_q$-definable minimal fields in the sense of $L_q$ and ACFA for all $q \in \qq$.
     \end{itemize}
\end{thm}

 This begs the question of what happens when $p_1$ is grouplike, i.e. nonorthogonal to a generic type of a minimal modular group. It is certainly possible for the rank to go up: for example, $\sigma_1(x) = x^4$ has an $L_{\frac{1}{2}}$-definable infinite, infinite-index subgroup $\sigma_{\frac{1}{2}}(x) = x^2$. However, if rank explodes, the algebraic group responsible for this must have something very close to a compositionally divisible quasiendomorphism, which seems unlikely.

\begin{conj} \label{grouconj}
 ($\mathbb{Q}$ACFA) Suppose that $p$ is an $L_\mathbb{Q}$-type whose reduct $p_1$ to $L_1$ is minimal in the sense of $L_1$ and ACFA. If $p_1$ is grouplike in  the sense of $L_1$ and ACFA, then $p$ has finite rank in the sense of $\mathbb{Q}$ACFA. \end{conj}

In Section \ref{grpsec}, we prove some special cases of this conjecture. In Proposition \ref{happygpprop}, we prove it for arbitrary subgroups of the multiplicative group and of elliptic curves fixed by all $\sigma_q$ in characteristic zero. Generalizing our proof to arbitrary simple abelian varieties fixed by all $\sigma_q$ would require thinking through some linear algebra over their endomorphism rings, which are usually not commutative. Abelian varieties that are not fixed by all $\sigma_q$ might even be easier to deal with. In ACFA in positive characteristic, the additive group of the field has minimal one-based subgroups; we have no idea what happens to these. In a different direction, Proposition \ref{grdeghappy} gives a soft proof of many cases of Conjecture \ref{grouconj}, relying only on the degrees of the algebraic group correspondence encoding the $L_1$-minimal group.

In Section \ref{sec2}, we prove our general model-theoretic results: Theorems \ref{onepurthm}, \ref{twopurthm}, and \ref{geopurthm}.
In Section \ref{sec3}, we recall some facts about ACFA, apply the results of Section \ref{sec2} to $\mathbb{Q}$ACFA to prove Theorem \ref{firstqacfathm} and develop some $\mathbb{Q}$ACFA-specific technical tools to prove Theorem \ref{better-qacfa}.

The reader should have a working knowledge of first-order model theory as in, for example, Hodges's Shorter Model Theory \cite{babyHodges}, whose notation we follow somewhat faithfully. Additionally, we assume the familiarity with basic notions of stability and simplicity theory that can be found, among other places, in \cite{casasimp} and \cite{bigpillay}. The discussion of rosy theories,  thron-forking, and abstract independence notions in Section \ref{sec2} follows \cite{adlerth} and is mostly irrelevant for applications to $\mathbb{Q}$ACFA. The language of naive algebraic geometry (see, for example, the first chapter of Hartshorne's \cite{hshor}) is used throughout Sections \ref{sec3} and \ref{grpsec}.

We thank Zoe Chatzidakis, Martin Hils, and Thomas Scanlon for long productive conversations about this paper, and several attentive seminar audiences (Paris 7, Maryland, UIC) that have helped us clarify the statements and proofs in this paper.

\section{Some pure model theory} \label{sec2}
\label{puresect}

\subsection{Notation, conventions, references.}

  A \emph{formula} has no parameters unless it is a formula over a parameter set, or a formula in a type over a set, or somesuch. No notational distinction is made between singletons and finite tuples of variables or elements of the model, unless explicitly stated otherwise.
  \emph{Signatures} are sets of symbols; \emph{languages} are sets of formulae.

  We work in multisorted first-order logic, and a subsignature may have fewer sorts. For example, if a theory $T$ is the reduct of a complete theory $S$, then $T^{eq}$ is a reduct of $S^{eq}$. Thus an element of the universe of a structure might no longer be in the universe of the reduct.


 We freely use the word ``Lascar rank'' and the notation $U(\pi)$ to denote various generalizations: the ambient theory need not be stable, and $\pi$ need not be a complete type; see Section \ref{geostabsec} for details. A partial type $\pi$ is \emph{minimal} if $U(\pi) = 1$.

 Our background references are: Hodges's \cite{babyHodges} for basic model theory, Shelah's \cite{shbible} for combinatorial approach to stability, Casanovas' \cite{casasimp} for simple theories, Onshuus' \cite{alf06} for rosy theories, and Adler's \cite{adlerth} for notions of independence.

\subsection{Setup and pre-stability}

  Recall that a collection $C$ of subsignatures of $L$ is \emph{sufficient} if every finite subset of $L$ is contained in an element of $C$. If, in addition, for all $L, L' \in C$ there is some $L'' \in C$ such that $L \cup L' \subset L''$, we say that the collection is \emph{directed}.
A sufficient collection of finite subsignatures is automatically directed. Sometimes, one starts with $L$ and seeks a sufficient collection $C$ of subsignatures. Conversely, a collection $C$ of signatures is a sufficient collection of subsignatures of $L := \cup C$ if and only if every finite subset of $\cup C$ is contained in some element of $C$.
\textbf{From now on, $\{ L_w \sthat w \in W \}$ is a sufficient collection of subsignatures of $L_W := \cup_{w \in W} L_w$.}
 Similarly, one might start with an $L_W$-theory $T_W$ and consider the reducts $T_w$ of $T_W$ to $L_w$; or one might start with a collection of $L_w$-theories $T_w$.

\begin{lem} \label{glue-theories}
 Suppose that for each $w \in W$ we have an $L_w$-theory $T_w$, and that for all $u, v \in W$ such that $L_u \subset L_v$, the reduct of $T_{v}$ to $L_u$ is precisely $T_u$. Let $T_W := \cup_{w \in W} T_w$. Then for each $w \in W$ the reduct of $T_W$ to $L_w$ is precisely $T_w$. \end{lem}

\textbf{From now on, $T_W$ is an $L_W$-theory, and $T_w$ are the reducts of $T_W$ to $L_w$.}
Similarly, given an $L_W$-structure $M_W$, we denote the reduct of $M_W$ to $L_w$ by $M_w$, $\acl_W$ and $\acl_w$ are algebraic closure in the (reducts to the) corresponding signatures, and $\inde^w$ is an independence notion (usually, non-forking) in the reduct. The subscript $W$ for the full signature is often dropped.

This is an extremely tame notion of ``limit theory'', where the properties of $T_W$ are very tightly controlled by the properties of $T_w$'s. Chris Laskowski has some results about stable theories $T_W$ that can be obtained from a collection of superstable theories $T_w$ by this construction.
 Taking this sort of limit clearly commutes with the construction of $M^{eq}$.

\begin{lem} Suppose that $T_W$ is a complete theory, and let $L'_W$ be the signature of $(T_W)^{eq}$ and $L'_w$ be the signatures of $(T_w)^{eq}$ for each $w \in W$. Then $\{ L'_w \sthat w \in W \}$ is a sufficient collection of subsignatures of $L'_W$, the reduct of $(T_W)^{eq}$ to $L'_w$ is precisely $(T_w)^{eq}$, and $\cup_w (T_w)^{eq} = (T_W)^{eq}$. \end{lem}

Properties that are $\forall \exists$ in formulae and properties that are evaluated one formula at a time ``pass to this limit'' in the sense that they hold $T_W$ whenever they hold in all $T_w$.

\begin{prop}
 For each of the following properties, if all $T_w$ have it, then $T_W$ also has it. List of properties: consistency, axiomatizations, completeness, quantifier elimination, partial quantifier elimination such as model-completeness, elimination of imaginaries, stable embeddedness of some definable set, characterization of algebraic closure...\end{prop}

\begin{proof}
Let us say more precisely what we mean by some of the properties. \begin{itemize}
\item If a set $A_w$ of $L_w$-sentences axiomatizes $T_w$ for each $w$,\\ then $A_W := \cup_{w \in W} A_w$ axiomatizes $T_W$.
\item If a set $\Delta_w$ of $L_w$-formulae is an elimination set for $T_w$ for each $w$, then $\Delta_W := \cup_{w \in W} \Delta_w$ is an elimination set for $T_W$.
\item If each $T_w$ has elimination of imaginaries, then so does $T_W$.
\item If for each $w$, $B_w$ is a set of $L_w$-formulae such that in models of $T_w$, algebraic closure is always witnessed by some formula from $B_w$; then $B := \cup_{w\in W} B_w$ does the same for $T_W$.
\end{itemize}

Recall that $\Delta$ is an \emph{elimination set} for $T$ if each formula is equivalent to some formula from $\Delta$ modulo $T$. For example, $T$ has quantifier elimination (resp. is model-complete) if and only if the set of quantifier-free (resp. existential) formulae is an elimination set for $T$.

All of these are equally obvious; for example and amusement, we prove the last. Take $a,b \in M \models T_W$ such that $b \in \acl_W(a)$. Let $\phi(x,y)$ be the formula witnessing this; it is an $L_w$-formula for some $w$, and $\phi(a, y)$ has finitely many solutions in $M_w$, so there is a formula $\psi(x,y) \in B_w$ with $M_w \models \psi(a,b)$ and  $\psi(a,y)^{M_w} = \psi(a,y)^M$ is finite. \end{proof}

\begin{prop} If $S_W$ is another $L_W$-theory such that $T_w$ is the model-companion of $S_w$ for each $w$, then $T_W$ is the model-companion of $S_W$.\end{prop}
\begin{proof}
 We already know that $T_W$ is model-complete, because all $T_w$ are, and model-completeness is equivalent to quantifier elimination down to existentials. To show that every model $M$ of $S_W$ embeds into some model of $T_W$, consider the atomic $L_W$-diagram  $\Gamma$ of $M$, and let $\Sigma := \Gamma \cup T_W$. It is sufficient to show that $\Sigma$ is satisfiable. If $F \subset \Sigma$ is finite, then $F$ is a set of $L_w$-sentences for some $w \in W$, satisfiable by embedding the model $M_w$ of $S_w$ into some model of $T_w$. Exactly the same argument shows that every model of $T_W$ embeds into some model of $S_W$. \end{proof}

 \subsection{Combinatorial Stability} \label{combostabsec}
Many properties of theories are local in formulae, that is, they can be verified by looking at one formula $\phi(x,y)$ at a time. As long as the verification for each formula is sufficiently uncomplicated, these properties pass up from all $T_w$ to $T_W$.

\begin{lem} \label{combistabformula}
 If an $L_W$-formula $\phi(x; y)$ has the Order Property (resp., Independence Property; Tree Property) in $T_W$, then the same formula has the same property in $T_w$ for any $L_w$ containing all symbols of $\phi$.
\end{lem}
\begin{proof}
 For the order property, find a formula $\phi(x,y)$, a model $M_W$ of $T_W$, and parameters $\{a_i \sthat i \in \omega \}$ such that
 $ \pi_n :=  \{ \phi(x, a_i) \sthat i \leq n\} \cup \{ \neg \phi(x, a_i) \sthat i > n\}$
 is consistent for each $n \in \omega$. Then for any $L_w$ which contains the symbols in $\phi$, the reduct $M_w$ and the same parameters $\{a_i\}$ witness that the same formula $\phi(x,y)$ has the Order Property.

 For the Independence Property, the proof is identical, except that $\pi_n$ should be replaced by
 $\pi_S:=  \{ \phi(x, a_i) \sthat i \in S \} \cup \{ \neg \phi(x, a_i) \sthat i \not\in S \}$
for $S \subset \omega$.

 For the Tree Property (see Definition 2.19 in \cite{casasimp}), the set of parameters should be indexed by nodes of a tree, countably branching, of countable height; and the consistent $\phi$-types should be indexed by paths through the tree. Additionally, the $k$-Tree Property requires $k$-inconsistence of certain other $\phi$-types: for each node $s$ and any $k$ integers $i_0, \ldots i_{k-1}$, the formula $\wedge_{j \in k} \phi(x, a_{s\frown i_j})$ should be inconsistent. Still, this clearly passes to any reduct containing the symbols of $\phi$.
\end{proof}

Just as with quantifier elimination, Lemma \ref{combistabformula} also implies refinements of the next proposition. In particular, the observation that \emph{quantifier-free} stability is preserved will be useful later.

\begin{prop} All $T_w$ are stable (resp., dependent; simple) if and only if $T_W$ is. \end{prop}

\begin{proof}
Stability is the lack of Order Property, simplicity is the lack of Tree Property, dependence is the lack of the Independence Property. All three good properties obviously pass to all reducts. Lemma \ref{combistabformula} shows that each of the three bad properties passes to some reduct in our setting. \end{proof}

Similar proofs will work for any combinatorial property which is truly local in formulae, that is, which does not require the witnesses to be have the same type \emph{in the full signature}, nor to be indiscernible \emph{in the full signature}.

 Rosiness is not characterized by a Shelahian combinatorial property, as far as I can tell. However, unwrapping the characterization given in \cite{adlerth} in terms of local dividing ranks yields the following fact; see Appendix for the details of unwrapping. 

 \begin{fact} (\cite{adlerth}) \label{adler-rosy-fact}
 An $L$-theory $T$ is not rosy if and only if there are:
 a formula $\phi(\bar{x}; u \bar{v})$, an integer $k$, a model $M \models T$, and parameters $b^i, \bar{c}^i, \bar{d}^i$ and $b^{ij}$ for $i, j \in \omega$ in $M^{eq}$ such that \begin{itemize}
 \item $( \wedge_{i < k}\, \phi(\bar{x}; u_i \bar{v}) ) \wedge ( \wedge_{i\neq j <k}\, u_i \neq u_j )$
 is inconsistent;
 \item $\{ \phi(x, b^i, \bar{c}^i) \sthat i \in \omega \}$ is consistent;
\item $\operatorname{tp}_L( b^{ij}, \bar{d}^i / A_i) = \operatorname{tp}_L( b^{i}, \bar{c}^{i} / A_i)$
where $A_i := \{ b^i, \bar{c}^i \sthat j < i \}$; and
\item $b^{ij} \neq  b^{ij'}$ for all $j \neq j'$.
\end{itemize}
\end{fact}

 It may well be that the characterization in terms of another rank, related to equivalence relations, given in \cite{alfclif07}, is just as good for this purpose. 

\begin{prop} All $T_w$ are rosy if and only if $T_W$ is. \end{prop}
\begin{proof}
 One direction is easy: rosiness is known to pass to reducts. 
  For the other direction, suppose that $T_W$ is not rosy. Then there are: a model $M_W$ of $T_W$, an $L_W$-formula $\phi(\vec{x}, y, \vec{z})$, sorts $S_x, S_y, S_z$ in $(M_W)^{eq}$, and parameters $b^i, \bar{c}^i, \bar{d}^i$ and $b^{ij}$ in the appropriate sorts satisfying the four conditions in the fact above. Some $L_w$ contains $\phi$ and the formulae defining the equivalence relations whose quotients are the sorts $S_x, S_y, S_z$. The same $M_w$, $\phi(\vec{x}, y, \vec{z})$, $\{{c}_i \sthat i \in \omega \}$, and $\{ b_{ij} \sthat i,j \in \omega \}$ now witness that $T_w$ is not rosy: of the four requirements in the fact above, the third requirement is only easier in the reduct, and others are unchanged. \end{proof}


 Stability, simplicity, and rosiness are interesting properties because they are characterized by the presence of a reasonable notion of independence, so in effect we have shown that if each $T_w$ has a decent independence relation $\inde^w$ in the sense of \cite{adlerth}, then $T_W$ should have a decent independence relation $\inde^W$. It is natural to try to characterize $\inde^W$ in terms of the $\inde^w$s.

 The simplest kind of dependence is algebraic closure. As signature grows, algebraic closures grow. This can turn forking both on and off. If $A \inde_C B$ in the reduct, and $A$ falls into the algebraic closure of $B$ in the expansion, forking is ``turned on''. If $A \nind_C B$ in the reduct, and $B$ falls into the algebraic closure of $C$ in the expansion, forking is ``turned off''.

 So, the real quest is to characterize $A \inde^W_C B$ in terms of $A \inde^w_C B$ in the special case when $C \subset A, B$ and all three are algebraically closed in the full signature. The following is obvious in stable theories, via definability of types.

\begin{conj} Suppose that $M_W$ is a monster model of $T_W$, that $T_W$ and all $T_w$ are simple (resp., rosy) and $\inde^W$, $\inde^w$ are the forking (resp., thorn-forking) independence relations on $M_w$. Let $C \subset A, B \subset M_W^{eq}$ be algebraically closed in the sense of $T_W^{eq}$. Then
 $$ A \inde^W_C B \mbox{ if and only if } A \inde^w_C B \mbox{ for all } w$$ \end{conj}

Surprisingly, it is the left-to-right implication in the conjecture that is hard. The next lemma is the easy contrapositive of the right-to-left implication. The definitions of forking and dividing are from \cite{shbible}; the definitions of thorn-forking, thorn-dividing, and strong-dividing are from \cite{alf06}.

\begin{lem} \label{forkingoesdown}
Suppose that $M_W$ is a monster model of $T_W$, that $T_W$ and all $T_w$ are simple (resp., rosy) and $\inde^W$, $\inde^w$ are the forking (resp., thorn-forking) independence relations on $M_w$. Let $C \subset A, B \subset M_W^{eq}$ be algebraically closed in the sense of $T_W^{eq}$. If $ A \nind^W_C B$, then $ A \nind^w_C B$ for some $w \in W$. \end{lem}

\begin{proof}
 Since $ A \nind^W_C B$, there are some $a \in A$, $b \in B$, and $\phi(x,y)$ such that $\models \phi(a,b)$ and $\phi(x,b)$ (thorn-)forks over $C$. That is, $\phi(x,b)$ implies the (finite) disjunction of some $\psi_i(x, d_i)$, each of which (thron-)divides over $C$.

  In the case of dividing, this just means that for each $i$, there is an integer $k_i$ and an infinite set $\{ d_i^j \sthat j \in \omega \}$ of realizations of the $L_W$-type of $d_i$ over $C$ such that $\{ \psi_i(x, d_i^j) \sthat j\in \omega \}$ is $k_i$-inconsistent. This clearly remains true in any $L_w$ which contains (all the symbols in) all $\psi_i$ and $\phi$.

In the case of thorn-dividing, this means that each $\psi_i(x, d_i)$ strong-divides over $C e_i$ for some $e_i$. That is, for each $i$, $d_i$ is not $L_W$-algebraic over $C e_i$ and there are a formula $\theta_i(y, z_i)$, an integer $k_i$, and some $f_i \in C e_i$ such that and $\wedge_{j=1}^{k_i+1} ( \psi_i(x, y^j) \wedge \theta_i(y^j, f_i))$ is inconsistent and $\models \theta_i(d_i, f_i)$. This clearly remains true in any $L_w$ which contains (all the symbols in) all $\psi_i$, $\theta_i$, and $\phi$.

\end{proof}

 We include the details of this easy proof because we refer back to its details in the proof of Proposition \ref{geostabprop} below, and in order to point out the difficulties in proving the other direction of the conjecture. In the case of dividing, we need "an infinite set of realizations of the $L_W$-type of $d_i$ over $C$". In the case of thorn-dividing, we need $e_i$ with enough knowledge about $d_i$ to witness strong-dividing, but not so much as to make $d_i$ algebraic. Both of these are difficult to preserve in passing from a reduct to an expansion.

 The following fact proves the other direction of the conjecture in the case where all $T_w$ eliminate hyperimaginaries. This includes the case where all $T_w$ are stable and the case where all $T_w$ are supersimple, and is conjectured to include the case where all $T_w$ are simple. This fact also applies to many cases where the $T_w$ are only rosy.

\begin{fact} (Exercise 3.5 in \cite{adlerth}) \label{exer35} Suppose that $M$ is a big model of $T$, and $T'$ is a reduct of $T$. Suppose that $T$ and $T'$ are simple (resp., rosy), and forking (resp. thorn-forking) independence in $T'$ is a canonical independence relation in the sense of Definition 3.1 in \cite{adlerth}. Let $C \subset A, B \subset M^{eq}$ be algebraically closed in the sense of $T^{eq}$. If $ A \inde_C B$ in the sense of $T$, then $ A \inde_C B$ in the sense of $T'$.\end{fact}

 We include for amusement a partial result for the conjecturally non-existent case of simple theories not subject to the last proposition. Its proof, suggested by Martin Hils, relies on the yoga of coheir sequences and Morley sequences, to be found, for example, in \cite{casasimp}.

\begin{prop} Suppose that $M$ is a big model of a simple $L$-theory $T$, and $T'$ is a reduct of $T$ to $L'$. Let $C \subset A, B \subset M^{eq}$ be algebraically closed in the sense of $T^{eq}$, and suppose further that $C$ is a model. If $ A \inde_C B$ in the sense of $T$, then $ A \inde_C B$ in the sense of $T'$.\end{prop}

\begin{proof}
 If not, there are $a \in A$, $b \in B$, and $L'$-formula $\phi(x,y)$ such that $\models \phi(a,b)$ and $\phi(x,b)$ $L'$-forks over $C$. That is, $\phi(x,b)$ implies the (finite) disjunction of $L'$-formulae $\psi_i(x, d_i)$, each of which $L'$-divides over $C$. It suffices to show that these $\psi_i(x, d_i)$ also $L$-divide over $C$; to lighten notation, we work with one of them and drop the subscripts.

 So: $\psi(x, d)$ is an $L'$-formula which divides over $C$ in the sense of $L'$. Let $\{ d^j \sthat j \in \omega \}$ be a non-constant $L$-coheir sequence in the $L$-type of $d$ over $C$. That is, each $d^j \equiv_{L,CD_j} d$ and the type of $d^j$ over $CD_j$ is finitely satisfiable in $C$, where $D_j := \{ d^{j'} \sthat j' < j \}$ . Then $\{ d^j\}$ is still an $L'$-coheir sequence in the $L'$-type of $d$ over $C$, so it is also $L'$-Morley sequence in the $L'$-type of $d$ over $C$. Since $\psi(x,d)$ divides over $C$ in the sense of $L'$, and all Morley sequences witness dividing, $\{ \psi(x, d^j) \sthat j \in \omega \}$ is $k$-inconsistent for some $k$. But all $d^j$ realize the $L$-type of $d$ over $C$, so this witnesses that $\psi(x,d)$ $L$-divides over $C$. \end{proof}

\subsection{Geometric Stability} \label{geostabsec}

Lascar rank, originally defined to be a property of complete types $p$ in stable theories, namely the foundation rank in the tree of forking extensions, is denoted by $U(p)$. It is natural to generalize it to other contexts (simple theories, rosy theories) where forking (or thorn-forking) works well. For a partial type $\pi$, such as a formula, $\sup \{ U(p) \sthat \pi \subset p \}$ provides a less robust but still useful notion of rank. Abusing notation, we call all of these generalizations ``Lascar rank'' and denote them by $U(\pi)$.

We similarly say that two partial types are nonorthogonal whenever some completions of them are nonorthogonal is the usual precise sense.

We state most results in this Section \ref{geostabsec} for directed collections of subsignatures; this is a purely cosmetic choice. Compare, for instance, the statement of the next proposition to the last sentence of its proof.

To lighten notation, we partially order $W$ by inclusion of $L_w$'s and write $u \leq v$ for $L_u \subset L_v$. It is easy to see that the limsup of Lascar ranks in $L_w$, if finite, is an upper bound on the Lascar rank in $L_W$.

\begin{prop}
 Suppose all $T_w$ and, therefore, $T_W$ are simple (resp. rosy), and let $n \in \omega$.
If the $L_W$-Lascar rank of an $L_W$-type $p$ is at least $n$, then there is some $w_0 \in W$ so that for all $w > w_0$ the $L_w$-Lascar rank of $p_w$ (the reduct of $p$ to $L_w$) is at least $n$. \end{prop}
\begin{proof}
 Since the $L_W$-Lascar rank of an $L_W$-type $p$ is at least $n$, there are $A_n \supset A_{n-1} \supset \ldots \supset A_0 = \operatorname{dom}(p)$ and $L_W$-types $p_i$ over $A_i$ such that $A_i = \acl_W(A_i)$ and $p_n \supset \ldots \supset p_0 = p$ and $p_{i+1}$ $L_W$-(thorn-)forks over $A_i$ for each $i$. As in the proof of Lemma \ref{forkingoesdown}, there is a finite list of formulae such that any $L_w$ that contain all the symbols in these formulae does the job.
\end{proof}

The following immediate corollary gives a practical way to look for $L_W$-types of low rank. An analogous but more cumbersome result holds when the collection of subsignatures is not directed.

\begin{sheep} \label{rankcomputesheep} Suppose that $\{ L_w \sthat w \in W \}$ is directed, $n \in \omega$, and $\pi$ is a partial $L_w$-type (such as an $L_w$-formula) such that for all $w' >w$ and for all $L_{w'}$-types $p_{w'} \supset \pi$, the $L_{w'}$-Lascar rank of $p_{w'}$ is at most $n$. Then the $L_W$-rank of every $L_W$-type that contains $\phi$ is at most $n$.
\end{sheep}

Some things can also be said about properties around the Zilber Trichotomy. 

\begin{deff} \label{deftrivmod}
 For a sufficiently saturated $L$-structure $M$ and a partial $L$-type $\pi$ over $A \subset M$,\begin{itemize}
\item  Algebraic closure is \emph{trivial} on the set of realizations of $\pi$, or $\pi$ is \emph{trivial}, when for any $k \in \nn$, if $a_k \in \acl(A, a_1, \ldots a_{k-1})$ and all $a_i$ are realizations of $\pi$, then $a_k \in \acl(A, a_i)$ for some $i < k$.

\item $\pi$ is \emph{one-based} if any two sets $B$ and $C$ of realizations of $\pi$ are independent over $\acl^{eq}(A \cup B) \cap \acl^{eq}(A \cup C)$.

\item $\pi$ is \emph{a modular group} when the set $G$ of realizations of $\pi$ admits a definable group structure, and any (relatively) definable subset of $G$ is a boolean combination of cosets of (relatively) definable subgroups of $G$.
\end{itemize} \end{deff}

Our definition of ``one-based'' is exactly that of ``modular'' in \cite{udizoe}. 
If $\pi'$ is a reduct of $\pi$, neither of $L$-triviality of $\pi$ and $L'$-triviality of $\pi'$ imply the other, and the same for one-basedness; but these properties do pass up to the limit in our sense.

\begin{prop} \label{geostabprop}
Suppose that $W$ is directed.
Let $\pi$ be a partial $L_W$-type, and suppose that there exists $w_0 \in W$ such that for all $w \geq w_0$, the reduct $\pi_w$ is $L_w$-trivial (resp., $L_w$-one-based, an $L_w$-modular group). Then $\pi$ is $L_W$-trivial (resp., $L_W$-one-based, $L_W$-modular group). \end{prop}

\begin{proof}
 With some compactness, triviality becomes $\forall\exists$ in formulae.

 For the second, suppose $B$ and $C$ witness that $\pi$ is not one-based. For each $w$, let $B^w := \acl^{eq}_w(AB)$, and similarly for $C^w$, and let $D^w := B^w \cap C^w$; and similarly for $B^W$, $C^W$, and $D^W$.
 Failure of one-basedness of $\pi$ in $L_W$ means that $B^W \nind^W_{D^W} C^W$.

 From the proof of Lemma \ref{forkingoesdown}, it follows that $B^w \nind^w_{D^w} C^w$ for some $w$, violating $L_w$-one-basedness, since realizations of $\pi$ are a fortiori realizations of $\pi_w$.

 Indeed, it suffices to take $L_w$ large enough to include all the symbols in the formulas $\phi$, $\psi_i$, and $\theta_i$ that appear in that proof, and large enough to ensure that the elements $b \in B^W$ and $c \in C^W$ and $f_i \in D^W e_i$ from that proof are still where they need to be: $b \in B^w$, $c \in C^w$, and $f_i \in D^w e_i$. Since $i$ ranges over a finite set in that proof, all this only requires a finite reduct of $L_W$, which is therefore contained in some $L_w$.




 The last is again $\forall\exists$ in formulae.
\end{proof}

\section{$\mathbb{Q}$ACFA} \label{sec3}

\subsection{Preliminaries and ACFA}

The natural signature for a field with one automorphism is $\{ +, \cdot, 0, 1, \sigma \}$, the signature of rings expanded by a unary function symbol $\sigma$ for the automorphism. The language of naive Weil-style algebraic geometry, as in the first chapter of \cite{hshor}, is convenient for describing definable sets in this setting. Thus, for us, an affine (resp., projective) \emph{variety} is a solution set of a finite set of (resp. homogeneous) polynomial equations in affine (resp. projective) space, and we always work over fields. Note that we do not require varieties to be irreducible. The \emph{algebraic locus of $a \in K \supset E$ over $E$} is the smallest variety defined over $E$ that contains $a$.

A rational function between varieties is \emph{dominant} (almost surjective) if its image is Zariski-dense in its target. A rational function is \emph{finite} if almost all fibers (more precisely, fibers above a Zariski-dense subset of its image) are finite. If $A$ is a variety, $A^\sigma = \{ \sigma(a) \sthat a \in A \}$ is defined by the same equations as $A$ but with coefficients twisted by $\sigma$. \textbf{When $A$ and $B$ are irreducible, $B \subset A \times A^\sigma$, and the two projections from $B$ to $A$ and to $A^\sigma$ are dominant and finite,} we write
$$(A,B)^\sharp := \{ a \in A \sthat (a, \sigma(a)) \in B \}.$$
Otherwise, we write $(A,B)^{sh} := \{ a \in A \sthat (a, \sigma(a)) \in B \}$ for the same set. The irreducibility and dominance hypotheses are harmless in that if they fail, $(A,B)^{sh}$ is actually trapped inside, and thus better understood in terms of, other smaller varieties. Without the finiteness hypothesis, the Lascar rank of $(A,B)^{sh}$ is infinite; we exclude these because we have nothing to say about them. When we work with many automorphisms, we write $(A,B)^{\sigma \sharp}$ to indicate the automorphism.

\begin{deff} Let $K$ be a field with an automorphism $\sigma$, and fix $F \subset K$ and $c \in K$.
We write $\langle F \rangle_{\sigma, \sigma^{-1}}$ for the field generated by $\cup_{n \in \mathbb{Z}} \sigma^n(F)$. The \emph{$\sigma$-degree of $c$ over $F$ in $K$} is 
$$ \deg_\sigma(c/F) := \operatorname{tr.deg.} ( \langle Fc \rangle_{\sigma, \sigma^{-1}} / \langle F \rangle_{\sigma, \sigma^{-1}})$$
If $A$ and $B$ are defined over $F$, a tuple $c \in (A, B)^\sharp$ is \emph{$\deg_\sigma$-generic in $(A, B)^\sharp$ over $F$} if $\deg_\sigma(c/F)$ is the dimension of $A$. A type is \emph{$\deg_\sigma$-generic} if some of its realizations in a monster model are $\deg_\sigma$-generic.
\end{deff}

In any case, the dimension of $A$ is always an upper bound on the $\sigma$-degree of elements of $(A, B)^\sharp$. It is easy to see that $c \in (A, B)^\sharp$ is $\deg_\sigma$-generic in $(A, B)^\sharp$ over $F$ if and only if the algebraic locus of $c$ over $F$ is precisely $A$; that is, a type is $\deg_\sigma$-generic if and only if the set of its realizations (in a sufficiently saturated model) is Zariski-dense in $A$.

We briefly summarize some of the results of \cite{udizoe} and \cite{udikobizoe}. The theory of fields (or integral domains, or algebraically closed fields) with an automorphism (or an injective endomorphism) has a model companion ACFA, axiomatized by the axioms for algebraically closed fields, the statement that $\sigma$ is an automorphism, and axioms requiring $(A,B)^\sharp$ to be Zariski-dense in $A$ for all $A$ and $B$. The completions of ACFA are given by specifying the characteristic of the field and the action of $\sigma$ on the algebraic closure of the prime field. If $E$ is a subset of a model of ACFA, the model-theoretic algebraic closure of $E$, denoted by $\acl_\sigma(E)$, is the field-theoretic algebraic closure of $\langle E \rangle_{\sigma, \sigma^{-1}}$. We write $E^{alg}$ for the field-theoretic algebraic closure of $E$. ACFA is supersimple, and forking-independence is given by
  $$E_1 \inde^{ACFA}_F E_2  \mbox{ if and only if }  \acl_\sigma(FE_1) \inde^{ACF}_{\acl_\sigma(F)} \acl_\sigma(FE_2) $$
  Thus, forking is always witnessed by quantifier-free formulae, which are stable. Indeed, forking formulas $\phi(x)$ are (or at least imply) ``new'' algebraic equations on $\{ \sigma^i(x) \}_{i \in \nn}$. It follows easily that when the $\sigma$-degree of $c$ over $F$ is finite, it is an upper bound on the Lascar rank of the type of $c$ over $F$. In particular, the algebraic dimension of $A$ is an upper bound on the Lascar rank of types in $(A,B)^\sharp$: any forking extension of a type in $(A,B)^\sharp$ must contain (or at least imply) new algebraic relations on $x$. This bound is rarely tight.

  Types of Lascar rank $1$ satisfy the Zilber Trichotomy.
 \begin{fact} (Zilber Trichotomy)
 \cite{udikobizoe} \label{zitrifact}
 Every complete type of Lascar rank $1$ in ACFA is exactly one of the following:
 \begin{itemize}
 \item \emph{disintegrated}: exactly as in Definition \ref{deftrivmod};
 \item \emph{grouplike}: one-based (see Definition \ref{deftrivmod}), and non-orthogonal to a generic type of a definable group of Lascar rank $1$;
 \item \emph{fieldlike}: non-orthogonal to a generic type of a field defined by $\tau(x) = x$, where $\tau = \sigma$ or, in positive characteristic, a composition of relatively prime powers of $\sigma$ and the Frobenius automorphism.
 \end{itemize}
 \end{fact}

One-based groups of rank $1$ are modular (see Definition \ref{deftrivmod}) in characteristic zero, but not necessarily in positive characteristic.

\subsection{$\mathbb{Q}$ACFA basics}

A field with an action by $(\zz, +)$ is a definitional expansion of a field with one automorphism, and
the formalism of Section \ref{puresect} is good for approximating a field with a $(\qq, +)$-action by fields with $(\zz, +)$-actions.

\begin{deff}
Let $L_\qq := \{ +, \cdot, 0, 1 \} \cup \{ \sigma_q \sthat q \in \qq\}$
 and for each $q \in \qq$, let
$L_q := \{ +, \cdot, 0, 1 \} \cup \{ \sigma_{nq} \sthat n \in \mathbb{Z} \} \subset L_\qq$,
where $\sigma_q$ are unary function symbols.

Let $S_\qq$ be the theory of fields with automorphisms $\sigma_q$ satisfying $\sigma_{q+r} = \sigma_q \circ \sigma_r$ for all $q, r \in \qq$.

Let $T_\qq := S_\qq \cup \{ ACFA_q \sthat q \in \qq \}$  where $ACFA_q$ is the axiomatization of ACFA for the automorphism $\sigma_q$. This is $\qq$ACFA.

\end{deff}

Clearly, $\{ L_q \sthat q \in \qq \}$ is a sufficient, directed collection of subsignatures of $L_\qq$.

\begin{lem} The reduct $S_q$ is axiomatized by ``this is a field'' and ``$\sigma_q$ is a field-automorphism'' and ``$\sigma_{nq} = \sigma_q^{\circ n}$'' and
``$(\emptyset)^{alg} \subset \operatorname{fix}(\sigma_q)$''.
\end{lem}
\begin{proof}
 First, $S_\qq$ implies all these: if $a$ is algebraic of degree $m$ over the prime field, then $\sigma_q(a) = a$ because $\sigma_{\frac{q}{m!}}$ fixes the prime field. Second, if $(M, \sigma_q)$ satisfies these, it can be expanded to a model of $S_\qq$ by taking lots of copies of $M$, freely amalgamating them over the algebraic closure of the prime field in $M$, and defining the new automorphisms to permute the copies in a coherent fashion.\end{proof}

\begin{lem} $T_\qq$ is consistent and $T_q$ is axiomatized by $\Gamma_q := S_q \cup ACFA_q$.\end{lem}
\begin{proof}
 By Lemma \ref{glue-theories}, it suffices to show that for any $q \in \qq$ and $m \in \zz^{>0}$, the reduct of $\Gamma_q$ to $L_{mq}$ is precisely $\Gamma_{mq}$. To see that $\Gamma_{mq}$ is contained in the reduct, note that $S_q$ clearly implies $S_{mq}$, and it is shown in \cite{udizoe} that $ACFA_q$ implies $ACFA_{mq}$. The completions of $\Gamma_{mq}$ are given by specifying the characteristic (since the action of the automorphism on the algebraic closure on the prime field is already specified), and $\Gamma_{q}$ does not specify the characteristic, so the reduct is no more than $\Gamma_{mq}$. \end{proof}

Thus, each $T_q$ is a definitional expansion of a completion of $ACFA$, and the general results of Section \ref{puresect} combine with properties of ACFA to yield the following.

\begin{prop}
$T_\qq$ is complete after specifying the characteristic; it is simple, quantifier-free stable, and is the model-companion of $S_\qq$.

 In models of $T_\qq$, model-theoretic algebraic closure of a set $A$ is the field-theoretic algebraic closure of $\cup_q \sigma_q(A)$.

 $T_\qq$ eliminates imaginaries, and eliminates quantifiers down to one existential quantifier over the algebraic closure.

 Suppose that $A \subseteq B, C$ and all three are algebraically closed in $L_\qq$; then the following are all equivalent\begin{itemize}
 \item $B \inde^{ACF}_A C$
 \item $B \inde^{q}_A C$ for some $q$
 \item $B \inde^{q}_A C$ for all $q$
 \item $B \inde^{\qq}_A C$
 \end{itemize}
\end{prop}
 \begin{proof} The last two are equivalent by Fact \ref{exer35}. Since algebraically closed sets remain algebraically closed in reducts, the equivalence of the first three follows from the characterization of non-forking in ACFA. \end{proof}

\begin{sheep}
 $$B \inde^\qq_A C  \mbox{ if and only if }  \acl_\qq(AB) \inde^{ACF}_{\acl_\qq(A)} \acl_\qq(AC) $$
\end{sheep}

The following observation makes the results of the pure section particularly easy to apply to $\qq$ACFA.

\begin{prop}
 For any $q_0 \in \qq$, the collection
 $ \{ L_q \sthat \exists n\in \nn\, nq = q_0\}$
  of subsignatures of $L_\qq$ is also sufficient. The collection $\{ L_{\frac{1}{n!}} \sthat n \in \nn^+ \}$ is also sufficient.
\end{prop}

\subsection{Fine Structure: $L_\qq$-rank and definable structure on $L_q$-minimal partial types.}

In this section, we use the results of Section \ref{geostabsec} to see what happens to a minimal type in ACFA as compositional roots of the automorphism are added to the signature. We first show that the ranks of the fixed fields of the named automorphisms explode, making $\qq$ACFA neither supersimple nor quantifier-free superstable.

\begin{deff} \label{fcapdef}
In a model of $\qq$ACFA, the fixed field of the named automorphism $\sigma_q$ is denoted by $F_q$.
The union and the intersection of all these are denoted by $F_\cup:= \cup_q F_q$ and $F_\cap := \cap_q F_q$.
\end{deff}

\begin{prop} \label{namedfieldprop}
The $L_\qq$-Lascar rank of $F_q$ is undefined. \end{prop}

\begin{proof}
These fields form a lattice: $F_{q} \subset F_{mq}$ for any integer $m$, and it follows easily from the axioms of ACFA (and is noted in \cite{udizoe}) that in a sufficiently saturated model, these extensions have infinite transcendence degree. An infinite descending chain of infinite extensions of definable fields gives rise to an infinite forking chain. \end{proof}

Since all non-algebraic $L_q$-types inside $F_q$ are nonorthogonal in the sense of $L_q$, the $L_\qq$-Lascar rank of any non-algebraic $L_{q'}$-type for any $q'$, such as an $L_\qq$-formula, inside these fixed fields is undefined.

As we noted above, all these $F_q$ contain the algebraic closure of the prime field. Since each of these is a fixed field of the automorphism of a model of ACFA, all are pseudofinite. It is easy to see that $F_q^{alg} \cap F_{mq}$ is precisely the unique extension of $F_q$ of degree $m$, and that $F_\cup:= \cup_q F_q$ and $F_\cap := \cap_q F_q$ are algebraically closed fields. From the fact that each $F_q$ is $L_q$-stably-embedded, it follows that the induced structure on $F_\cap$ is just the field structure, so its unique non-algebraic type has $L_\qq$-Lascar rank $1$. On the other hand, $F_\cup$ with all the automorphisms, or even with just the lattice of named subfields $F_q$, may be an interesting structure in its own right.

In positive characteristic, other definable automorphisms arise as compositions of $\sigma_q$ with powers of the Frobenius automorphism $\Phi$. Unlike the situation with one automorphism, where $(K, \Phi\circ \sigma)$ is just as good a model of ACFA as $(K, \sigma)$, adding compositional roots of $\sigma$ but not of $\Phi \circ \sigma$ introduces a real asymmetry. One might try to fix this by also adding compositional roots of $\Phi \circ \sigma$. However, at least if we maintain a requirement that all named automorphisms commute, this would also add compositional roots of the Frobenius itself. That is impossible, as the action of such roots on $\mathbb{F}_p^{alg}$ would have to be a root of $\hat{1}$ in the Pr\"{u}fer group $\hat{\zz}$. The next proposition shows the enormity of this asymmetry: while the Lascar rank of the $F_q$ explodes, the ranks of fixed fields of compositions of (powers of) the Frobenius with $\sigma_q$ remain finite.

\begin{prop} \label{finrakfieldprop}
For any $q_0 \in \qq$ and any integer $m \neq 0$,
$$U_\qq( \sigma_{q_0}(x) = \Phi^m(x)) = m.$$ \end{prop}

\begin{proof}
 Let $q_1 := {\frac{q_0}{m}}$, so that $( \sigma_{q_0}(x) = \Phi^m(x) )$ is equivalent to
 $ (\sigma_{q_1}^m(x) = \Phi^m(x))$, which is in turn equivalent to
 $( (\Phi \circ \sigma_{q_1}^{-1})^m (x) = x )$. Now in the reduct to $L_{q_1}$, this is simply the fixed field of the $m$th compositional power of the automorphism $(\Phi \circ \sigma_{q_1}^{-1})$ of a model of ACFA, so by \cite{udizoe} it has Lascar rank $m$, coming from an $m$-step analysis where each step is the fixed field of $(\Phi \circ \sigma_{q_1}^{-1})$. Thus, is suffices to show that $U_\qq( \sigma_{q_1}(x) = \Phi(x)) = 1$; that is, to prove the proposition for $m=1$.

 Since $ \{ L_q \sthat \exists n\in \nn\, nq = q_1\}$ is a sufficient collection of subsignatures of $L_\qq$, by Corollary \ref{rankcomputesheep} it suffices to show that for every $n \in \nn^+$,
 $U_{q_2}( \sigma_{q_1}(x) = \Phi(x)) = 1$ where $q_2 := \frac{q_1}{n}$. This is proved in \cite{udizoe}, as the reduct to $L_{q_2}$ (with automorphism $\sigma_{q_2}$) is a model of ACFA, and $n$ and $1$ are relatively prime. \end{proof}

 While in ACFA we must consider the fixed fields of $\sigma^m \Phi^n$ for various relatively prime $(m,n)$, in $\qq$ACFA the last two propositions take care of all minimal definable fields, as $\sigma_q^m$ is just another $\sigma_{q'}$. We now turn to the other two cases of the Zilber Trichotomy (Fact \ref{zitrifact}).

 The last paragraph of the last proof exemplifies our general approach to the study of the $L_{\qq}$-structure on the set of realizations of a partial $L_q$-type: we note that $ \{ L_{\frac{q}{n}} \sthat  n \in \mathbb{N}^{>0} \}$ is a sufficient collection of subsignatures of $L_\qq$, and that Corollary \ref{rankcomputesheep} and Proposition \ref{geostabprop} allows us to work with one of these at a time.
  The following lighter notation for considering these two automorphisms, including two notions of prolongation, is used heavily in the rest of this section, and also in Section \ref{grpsec}.

 \begin{notass} (In force until the end of this section.) \label{notasigtau}

 \begin{itemize}
 \item $ (K, (\sigma_q)_{q \in \qq})$ is a sufficiently saturated model of $\qq$ACFA.
 \item $E = \acl_\qq(E) \subset K$ be a subfield of $K$ that is algebraically closed in the full signature.
 \item Fix $q \in \qq^\times$ and $n \in \nn^{>0}$; let $\sigma := \sigma_{q}$ and $\tau := \sigma_{\frac{q}{n}}$. We use subscripts $\sigma$ and $\tau$ instead of $q$ and $\frac{q}{n}$.
 \item $\pi$ is a partial $L_\sigma$-type over $E$, and $S$ is the set of its realizations in $K$.
 \end{itemize}
 For any set $X \subset K$ and any element $a \in K$, let
 $$X^\prodplus := X \times \tau(X) \times \ldots \times \tau^{n-1}(X)$$
 $$X^\uniplus := X \cup \tau(X) \cup \ldots \cup \tau^{n-1}(X)$$
 $$a^+ := (a,  \tau(a), \ldots \tau^{n-1}(a))$$
so if $a \in X$, then $a^+ \in X^\prodplus$. Occasionally, we abuse notation and write $a^+ \subset X^\uniplus$.\end{notass}

 We begin with some easy observations.

\begin{lem} \label{babyprolonglem}
\begin{enumerate}
\item $\acl_\tau(a) = \acl_\sigma(a^+)$.
\item $\acl_\tau(X) = \acl_\sigma(X^\uniplus)$.
\item $a \in S$ if and only if $a^+ \in S^\prodplus$, but most elements of $S^\prodplus$ are not of this form.
\item $\tau$ is an automorphism of $(K, \sigma)$.
\item $\tau^i ( S )$ has the same $L_\sigma$ properties (Lascar rank, Zilber Trichotomy classification, etc) as $S$.
\end{enumerate}
\end{lem}

Any type of finite $\sigma$-degree is interdefinable with one to which the next two lemmas apply. For example, the lemmas apply whenever $\pi$ contains a formula defining $(A, B)^{\sigma \sharp}$ for some algebraic varieties $A$ and $B$ defined over $E$.

\begin{lem} \label{1stypelem}
With Notation and Assumptions \ref{notasigtau},
suppose that for any $a \models \pi$, $\sigma(a) \in (Ea)^{alg}$ and $\sigma^{-1}(a) \in (Ea)^{alg}$, and let $A$ be the Zariski closure of $S$ over $E$. 

If $a$ $\tau$-forks with some $F = \acl_\qq(F)$ over $E$, then the algebraic locus of $a^+$ over $F$ is a proper subvariety of the locus of $a^+$ over $E$, and in particular a proper subvariety of $A^{\times +}$.

If (some completion of) $\pi$ is $\tau$-nonorthogonal to a $\tau$-minimal $\tau$-type $p$, then $\pi \vee \tau(\pi) \vee \ldots \vee \tau^{n-1}(\pi)$ is $\sigma$-nonorthogonal to the reducts of $p$ and of $p \times \tau(p) \times \ldots \times \tau^{n-1}(p)$ to $L_\sigma$. \end{lem}

\begin{proof}
If $a$ $\tau$-forks with some $F = \acl_\qq(F)$ over $E$, then $\acl_\tau(Ea)$ is not field-independent from $F$ over $E$. Now $a^+$ contains a transcendence basis for $\acl_\tau(Ea)$ over $E$, which must now not be algebraically independent over $F$. The variety $A^{\times +}$ defined over $E$ contains $S^{\times +}$.

If $\pi$ is $\tau$-nonorthogonal to a $\tau$-minimal $\tau$-type $p$, then (perhaps after a base extension) some realization $b$ of $p$ is $\tau$-algebraic over $S$. Then $b^+$ is also $\tau$-algebraic over $S$, and so both $b$ and $b^+$ are $\sigma$-algebraic over $S^{\cup +}$, the set of realizations of $\pi \vee \tau(\pi) \vee \ldots \vee \tau^{n-1}(\pi)$. \end{proof}

\begin{lem} \label{2ndtypelem}
With Notation and Assumptions \ref{notasigtau},
suppose that for any $a \models \pi$, $\sigma(a) \in (Ea)^{alg}$ and $\sigma^{-1}(a) \in (Ea)^{alg}$, and let $A$ be the Zariski closure of $S$ over $E$.
\begin{enumerate}

\item If $\acl_\sigma$ is trivial on $S$, then so is $\acl_\tau$.

\item If $\acl_\sigma$ is trivial on $S$ and $\pi$ is $\sigma$-minimal, then $\pi$ is also $\tau$-minimal.

\item If $\pi$ is $\sigma$-orthogonal to $\fix(\sigma)$ (respectively, to all $\sigma$-definable fixed fields), then $\pi$ is also $\tau$-orthogonal to $\fix(\tau)$ (respectively, to all $\tau$-definable fixed fields).
\end{enumerate} \end{lem}

\begin{proof}
For the first part, suppose that $b \in S$, $C \subset S$, and $b \in \acl_\tau(C)$; we need to show that $b \in \acl_\tau(c)$ for some $c \in C$. Clearly, $C^\uniplus \subset S^\uniplus$ and $b \in S \subset S^\uniplus$. Now $S^\uniplus$ is the finite union of $L_\sigma$-definable, $L_\sigma$-minimal, and $L_\sigma$-trivial sets $\tau^i(S)$, so it is itself $L_\sigma$-definable, $L_\sigma$-minimal, and $L_\sigma$-trivial. By Lemma \ref{babyprolonglem}, $\acl_\sigma(C^\uniplus) = \acl_\tau(C) \ni b$. Thus, by $\sigma$-triviality of $S^\uniplus$, we have $b \in \acl_\sigma(d)$ for some $d \in C^\uniplus$. Now $d = \tau^i(c)$ for some $i$ and some $c \in C$, and so $b \in \acl_\tau(c)$, as wanted.

For the second part, by the first part of Lemma \ref{1stypelem}, it suffices to show  that all varieties $V \subsetneq A^\prodplus$ such that $S^\prodplus \cap V$ is infinite are defined over $E$. Indeed, any such $V$ defines an interesting algebraic relation among realizations of $\tau^i(\pi)$, each of which is $L_\sigma$-minimal and $L_\sigma$-trivial. Therefore, $V$ must be essentially binary: a component of the intersection of $V_{jj'}$ each of which witnesses $L_\sigma$-nonorthogonality between nonalgebraic types $q_j \in \tau^j(S)$ and $q_{j'} \in \tau^{j'}(S)$.
 As all $\tau^i(S)$ are $L_\sigma$-minimal and $L_\sigma$-trivial, such interalgebraic relations between them do not come in families; that is, there are only finitely many parameters that give such a thing for any particular formula. Thus, these parameters are in the model-theoretic $L_\sigma$-algebraic closure of any set over which both $\tau^{j'}(S)$ and $\tau^{j}(S)$ are defined, such as $E$.

For the last part, suppose towards contradiction that $\pi$ is $\tau$-nonorthogonal to the fixed field of $\tau^{\circ m} \circ \Phi^{\circ m'}$, where $m \in \zz^\times$, $m' \in \nn$, and $\Phi$ is the Frobenius automorphism. In the model $(K, \tau)$ of ACFA, fixed fields are analyzable in minimal fixed fields, so we may assume that this field is minimal. Now by the second part of Lemma \ref{1stypelem},
$\pi^{\times +}$ is $\sigma$-nonorthogonal to the $\sigma$-reduct of the fixed field of $\tau^{\circ m} \circ \Phi^{\circ m'}$, which is inside the fixed field of $(\tau^{\circ m} \circ \Phi^{\circ m'})^{\circ n} = \sigma^{\circ m} \circ \Phi^{\circ m'n}$. If $m' = 0$, this makes $\pi$ $\sigma$-nonorthogonal to $\fix(\sigma^{\circ m})$ and, therefore, to $\fix(\sigma)$; in any case, this makes $\pi$ $\sigma$-nonorthogonal to some fixed field. \end{proof}

All parts of Theorem \ref{better-qacfa} have now been proved.

\begin{proof} (Proof of Theorem \ref{better-qacfa}.) \label{pfthm5}
The first two parts of Lemma \ref{2ndtypelem} combine with Corollary \ref{rankcomputesheep} and Proposition \ref{geostabprop} to show that $L_q$-minimal $L_q$-trivial partial $L_q$ types remain minimal trivial in the full signature $L_\qq$. Propositions \ref{finrakfieldprop} and \ref{namedfieldprop} describe what happens to $L_q$-minimal fields, and the last part of Lemma \ref{2ndtypelem} shows that orthogonality to these fields is also preserved. \end{proof}


\section{Groups in $\mathbb{Q}$ACFA.} \label{grpsec}

 While the Lascar rank of the $\sigma$-minimal $\sigma$-trivial $(A, B)^{\sigma \sharp}$ remains $1$ as compositional roots of $\sigma$ are added to the language, its (difference) Krull dimension may very well go up. For example, if $n=2$ and $B$ is the graph of $f^\tau \circ f$ for some function $f: A \rightarrow A^\tau$, the graph of $f$ viewed as a subvariety of $A \times \tau(A)$ witnesses that the $\tau$-Krull dimension of $(A, B)^{\sigma \sharp}$ is at least $2$. We expect that this does not happen for generic varieties $A$ and $B$ (so that their Krull dimension remains $1$, even in the full signature), and that this should not ever happen infinitely often (so that the Krull dimension remains defined for minimal trivial $(A, B)^{\sigma \sharp}$, even in the full signature). At first glance it seems that the degrees of the correspondences give a bound, but the fiber product of several iterations of a correspondence might be reducible, and its components might have lower degrees, not even dividing the degrees of the reducible fiber product, so that doesn't work  - unless everything takes place in the category of algebraic groups.
 In Section \ref{gpdegsec}, we work out the details of this idea, culminating with Proposition \ref{groupdegreelem} that proves many cases of Conjecture \ref{grouconj}.

 For groups, the same issue actually increases the Lascar rank. For example, with $n=2$, the $\sigma$-grouplike $\sigma$-minimal group $G$ defined by the equation $\sigma(x) = x^9$ has an infinite, infinite-index subgroup $H$ defined by $\tau(x) = x^3$. The uniformly $L_\tau$-definable cosets of $H$ partition $G$ into infinitely many infinite subsets, making the $L_\tau$ Lascar rank of $G$ at least $2$. This happened because not only is the degree of the endomorphism $x \mapsto x^9$ a square of an integer, but the endomorphism itself has a compositional root $x \mapsto x^3$. In Section \ref{lastsec}, we prove more cases of Conjecture \ref{grouconj} by verifying that most (quasi)endomorphisms of nice algebraic group have few compositional roots.

 Before we recall the structure of $\sigma$-definable $\sigma$-minimal $\sigma$-grouplike groups worked out in \cite{hrumamu}, we make a simple observation that will simplify our bookkeeping.

\begin{lem} \label{nonorthinv}
With Notation and Assumptions \ref{notasigtau}, suppose that $p$ and $q$ are $L_\sigma$-minimal, $L_\sigma$-nonorthogonal $L_\sigma$-types. Then $p$ and $q$ have the same $L_\tau$-rank and the same $L_{\mathbb{Q}}$-rank.
 
 In particular, it suffices to prove Conjecture \ref{grouconj} for generic types of actual $L_\sigma$-minimal groups, rather than arbitrary grouplike types.
  
  In another particular, generic types of an $L_\sigma$-definable group $G$ have the same $L_\tau$-rank and the same $L_{\mathbb{Q}}$-rank as generic types of finite-index subgroups of $G$, finite-index group extensions of $G$, and finite-kernel quotients of $G$.\end{lem}

 \begin{proof}
 Nonorthogonality between $p$ and $q$ is witnessed by an $L_\sigma$-definable finite-to-finite correspondence between the sets of their realizations. This correspondence is, of course, still definable in any expansion $L^+$ of $L_\sigma$, such as $L_\tau$ or $L_{\mathbb{Q}}$. Lascar rank is preserved by definable finite-to-finite correspondences, so $p$ and $q$, now viewed as partial $L^+$-types, have the same Lascar rank in $L^+$.

 By the Zilber Trichotomy for ACFA, in the reduct to $L_\sigma$, any minimal grouplike type $q$ is nonorthogonal to a generic type $p$ of a one-based minimal definable group.
\end{proof}

As before, Corollary \ref{rankcomputesheep} and Proposition \ref{geostabprop} allow us to do all the hard work in finite reducts $L_q$ of $L_\qq$, which is to say in ACFA. We summarize the thorough treatment of one-based groups in ACFA in \cite{hrumamu} in Sections \ref{group-redct-sect} and \ref{quaends}, and give some proofs in order to introduce the notation and the intuitions behind it.

\subsection{Encoding minimal groups in ACFA in terms of algebraic groups.} \label{group-redct-sect}

 \begin{fact} \label{whoisa0}
 \textbf{Every $L_\sigma$-definable, $L_\sigma$-minimal, $L_\sigma$-one-based group $G_0$ is, up to finite-index subgroups and finite-kernel images, a Zariski-dense subgroup of a simple algebraic group $A_0$.}
 \begin{enumerate}
 \item Up to nonorthogonality (finite-index subgroups and finite-kernel images), any definable one-based group $G_0$ is a subgroup of a commutative algebraic group $A_0$. Without loss of generality, $G_0$ is Zariski-dense if $A_0$.
 \item Without loss of generality, the algebraic group $A_0$ is simple in the sense that there is no infinite algebraic subgroup $C$ of $A_0$ of infinite index, that is with infinite quotient $A_0/C$. The reason is that such a subgroup would give rise to a subgroup $H = C \cap G_0$ of $G_0$ and to the quotient $G_0/H$. If one of these is finite, the other is a minimal one-based group nonorthogonal to $G_0$, contained in a lower-dimensional algebraic group $C$ or $A_0/C$. If both $H$ and $G_0/H$ are infinite, $G_0$ cannot be minimal.
 \item Without loss of generality, the algebraic group $A_0$ is connected: the intersection $H$ of $G_0$ with a finite-index subgroup of $A_0$ would be a finite-index subgroup of $G_0$, thus nonorthogonal to $G_0$, contained in a same-dimensional, lower-degree algebraic group.
 \item So, the only algebraic subgroups of $A_0$ are finite.
 \end{enumerate} \end{fact}

\begin{fact} \label{whoisb0}
\textbf{Whenever an infinite $L_\sigma$-definable subgroup $G_0$ of a simple algebraic group $A_0$ has finite $L_\sigma$-Lascar rank, it is encoded by algebraic groups $B_0$, $A$, and $B$ as follows.}
 Let $$G_0^{[x]} := \{ (g, \sigma(g) \sigma^2(g), \ldots \sigma^{x}(g)) : g \in G_0 \} \subset
  A_0 \times A_0^\sigma \times \ldots \times A_0^{\sigma^x}\mbox{, and let}$$
 $$ X := \{ x \in \mathbb{N} :  G_0^{[x]}\mbox{ is not Zariski-dense in }A_0 \times A_0^\sigma \times \ldots \times A_0^{\sigma^{x}}.$$
 Since $G_0$ has finite $L_\sigma$ rank, $X$ is non-empty. Since $G_0$ is infinite, it must be Zariski-dense in $A_0$, so $0 \notin X$. Let $m$ be the least integer in $X$; and let $B_0$ be the Zariski closure of $G_0^{[m]}$, an algebraic subgroup of $A_0 \times A_0^\sigma \times \ldots \times A_0^{\sigma^m}$.

\label{whoareABG}
  Let $A := A_0 \times A_0^\sigma \times \ldots \times A_0^{\sigma^{m-1}}$;
  let $$B = \{ (b_0, b_1, \ldots b_{m-1}; b_1, b_2, \ldots b_m) : (b_0, b_1, \ldots b_m) \in B_0\}$$ be the subgroup of $A \times A^\sigma$ naturally obtained from $B_0$ by repeating all but the first and last coordinates;
  and let $G := (A, B)^{sh}$. All assumptions required for the notation $(A, B)^\sharp$ follow from the choice of $m$, except that $B$ might not be irreducible. In any case, $G_0^{[m-1]}$, a definably isomorphic copy of $G_0$, and $G$ have the same connected component, so the generic types of $G$ and $G_0$ have the same Lascar rank in any expansion, so it suffices to work with $G$ instead of $G_0$.
\end{fact}

 \textbf{Side notes.} \begin{itemize}
 \item We do not use our spiffy $A_0^{\times +}$ notation here because we reserve it for $\tau$-prolongations rather than $\sigma$-prolongations.
 \item The last step, passing from $G_0$ to $G$, might undo some of the reductions from the first step, where we passed to finite-index subgroups of $G_0$. This is fine: the only goal of the reductions in the first step was to obtain the simple algebraic $A_0$.
 \item Similarly, unlike $A_0$, the new group $A$ is not a simple algebraic group; but it is a product of simple algebraic groups $A_0^{\sigma^k}$, which is what we exploit later.
 \end{itemize}

\vspace{.5cm}

\subsection{Some cases of Conjecture \ref{grouconj} follow from degree computations.} \label{gpdegsec}

\begin{deff}
For irreducible algebraic groups $A$ and $B$,
an \emph{algebraic group correspondence from $A$ to $B$} is a (possibly reducible) subgroup $C \leq A \times B$ such that the projections $\pi: C \rightarrow A$ and $\rho: C \rightarrow B$ are finite dominant morphisms.

The \emph{degree ratio} of such an algebraic group correspondence is $\frac{\deg(\rho)}{\deg(\pi)}$.
\end{deff}

For example, if $C$ is the graph of an algebraic group homomorphism $f$ from $A$ to $B$, then this degree ratio is the degree of $f$.

\begin{lem}
Consider algebraic group correspondences $G \leq A \times B$ and $H \leq C \times D$ with degree ratios $r$ and $s$.
The degree ratio of the product $(G \times H) \leq (A \times C) \times (B \times D)$ is $rs$.
If $B=C$, the degree ratio of the composition (fiber product over $B$) $G \circ H \leq A \times D$ is also $rs$.
\end{lem}
\begin{proof} Immediate. \end{proof}

\begin{lem}
If $C$ is an algebraic group correspondence from $A$ to $B$, then its connected component $C_0$ is another algebraic group correspondence from $A$ to $B$ with the same degree ratio.
\end{lem}

\begin{proof}
Let $\pi_0$ and $\rho_0$ be the restrictions of $\pi: C \rightarrow A$ and $\rho: C \rightarrow B$ to $C_0$.

Since $C_0$ has the same dimension as $C$, the morphisms $\pi_0$ and $\rho_0$ are finite.
Since $\pi$ and $\rho$ are finite dominant, $A$, $B$ and $C$ (and, therefore, $C_0$) all have the same dimension.
Thus, the images of $\pi_0$ and $\rho_0$ have the same dimension as $A$ and $B$, so $\pi_0$ and $\rho_0$ are dominant, because $A$ and $B$ are irreducible.

To compare degree ratios, let $e$ be the index of $C_0$ in $C$; now the rest of the irreducible components of $C$ are the $(e-1)$ cosets of $C_0$. It follows that the degrees of the restrictions of $\pi$ and $\rho$ to any other irreducible component of $C$ are equal to the degrees of $\pi_0$ and $\rho_0$, respectively. Adding up, we get $e \deg(\pi_0) = \deg(\pi)$ and $e \deg(\rho_0) = \deg(\rho)$.
\end{proof}

\begin{lem} \label{groupdegreelem}
With Notation and Assumptions \ref{notasigtau}, consider a $\tau$-degree $m \lneq n$ subgroup of $(A,B)^{\sigma \sharp}$ defined by $\tilde{a} \in (\tilde{A}, C)^{\tau \sharp}$ where
$\tilde{a} := (a, \tau(a), \ldots \tau^{m-1}(a))$ and $\tilde{A} := A \times \tau(A) \times \ldots \times \tau^{m-1}(A)$.
If $x$ is the degree ratio of $B \leq A \times A^\sigma$ and $y$ is the degree ratio of $C \leq \tilde{A} \times \tau(\tilde{A})$, then $x^m = y^n$.
\end{lem}

\begin{proof}
The connected components $B_0$ of $B \times \tau(B) \times \ldots \times \tau^{m-1}(B)$ and $C_0$ of
$\tau^{n-1}(C) \circ \ldots \circ \tau(C) \circ C$ are both subgroups of $\tilde{A} \times \sigma(\tilde{A})$ and must be equal. The degree ratio of $B_0$ is $x^m$ and the degree ratio of $C_0$ is $y^n$.
\end{proof}

\begin{prop} \label{grdeghappy}
With Notation and Assumptions \ref{notasigtau}, let $A$ and $B$ be algebraic groups such that $(A,B)^{\sigma \sharp}$ is a $\sigma$-minimal, $\sigma$-one-based group.
If the degree ratio $x_0$ of $B \leq A \times A^\sigma$ is not $1$, then the $L_{\mathbb{Q}}$ rank of $(A,B)^{\sigma \sharp}$ is finite.

More precisely, the $L_{\mathbb{Q}}$ rank of $(A,B)^{\sigma \sharp}$ is at most $S$, the greatest integer for which $\sqrt[S]{x_0}$ is rational.
\end{prop}

\begin{proof}
By Corollary \ref{rankcomputesheep}, it suffices to obtain a bound on the $\tau$-rank of $(A,B)^{\sigma \sharp}$ that is independent of $n$.

By Proposition \ref{geostabprop}, $(A,B)^{\sigma \sharp}$ is $\tau$-one-based for all $\tau$. Thus, a $\tau$-forking chain of length $\ell$ will be witnessed by a chain of subgroups $G_\ell \leq G_{\ell -1} \leq \ldots \leq G_1 \leq G_0 = (A,B)^{\sigma \sharp}$. Since $(A, B)^{\sigma \sharp}$ is $\sigma$-minimal, all of these subgroups must be Zariski-dense in $A$.
Analysing each $G_i \leq A$ as in Fact \ref{whoisb0} shows that the pair $G_{i} \leq G_0$ satisfies the hypotheses of Lemma \ref{groupdegreelem} for each $i$.

Let $m_i$ be the $\tau$-degree of $G_i$, and let $x_i$ be the degree ratio of the algebraic group correspondence defining $G_i$. Then $x_0^{m_i} = x_{i}^{n}$. Since this chain of subgroups witnesses forking, the $\tau$-degree of $G_{i+1}$ is strictly lower than the $\tau$-degree of $G_i$. So we have integers $n=m_0 \gneq m_1 \gneq \ldots \gneq m_\ell \geq 1$ and rational numbers $x_i = x_0^{\frac{m_i}{n}}$.

Now for a positive rational number $x_0 \neq 1$, there are only finitely many rational numbers $\frac{r}{s}$ between $0$ and $1$ such that $x_0^{\frac{r}{s}}$ is rational: assuming that $r$ and $s$ are relatively prime, $x_0^{\frac{r}{s}}$ is rational if and only if $x_0^{\frac{1}{s}}$ is rational. Since $1$ is the only divisible element of the multiplicative group of positive rationals, this bounds $s$. Since $\frac{r}{s} < 1$, $s$ bounds $r$.

\end{proof}

This proves Conjecture \ref{grouconj} for $\sigma$-degree $1$ subgroups of the multiplicative group $\gm$, because irreducible subgroups of $\gm \times \gm$ are all of the form $x^m = y^n$ for relatively prime integers $m$ and $n$, whose degree ratio is not $1$ unless both $m$ and $n$ are $\pm 1$, in which case the group $\sigma(x) =  x^{\pm 1}$ is not one-based. For $\sigma$-degree $1$ subgroups of elliptic curves, this no longer suffices: the two projections from $B$ to $A$ and $A^\sigma$ may have the same degree without being the same map. Nor does this work for higher $\sigma$-degree subgroups of $\gm$: for example, the group correspondence from $\gm^2$ to itself encoding (as in Fact \ref{whoisb0}) the group defined by $\sigma^2(x) = \frac{\sigma(x)^4}{x}$ is a bijection, but the group is one-based. We return to summarizing results from \cite{hrumamu} to address these issues.

\subsection{Linear algebra with quasihomomorphisms.} \label{quaends}

Quasiendomorphisms are a standard tool for describing one-based groups. In general, a \emph{quasihomomorphism} from a group $X$ to a group $Y$ is a subgroup of $X \times Y$ for which the projection to $X$ is surjective and has finite fibers (equivalently, finite kernel $K$). If the kernel is trivial, this subgroup is the graph of an actual homomorphism. In any case, one may treat this subgroup as a ``finitely-valued function'' that, for a given input in $X$, returns several outputs from $Y$ instead of one; these outputs form a coset of $K$. Composition of quasihomomorphisms produces quasihomomorphisms, as fiber products of subgroups are subgroups.
For example, $B$ above is a quasiendomorhpism from $A$ to $A^\sigma$; we now describe it in terms of quasihomomorphisms among $\sigma$-transforms of $A_0$.

\begin{fact} \label{bijs}
 Let $e_0$ be the group identity of $A_0$.
 For $1 \leq i,j \leq m$, let $B_{ij}$ be the subgroup of $A_0^{\sigma^{i-1}} \times A_0^{\sigma^{j}}$ obtained by intersecting $B$ with
 $$(C_1 \times C_2 \times \ldots \times C_{m}) \times (D_1 \times \ldots \times D_m)\mbox{ where}$$
 $$C_i = A_0^{\sigma^{i-1}}, \,\,\, D_j = A_0^{\sigma^j}
    \mbox{, and } C_{k+1} = D_k = \{\sigma^k(e_0)\}\mbox{ for all other }k.$$
 Notation abuse alert: the definitions of $C_k$, $D_k$ depend on $i$ and $j$!!
 \begin{itemize}
 \item Each $B_{ij}$ is the quasihomomorphism from $A_0^{\sigma^{i-1}}$ to $A_0^{\sigma^{j}}$ obtained by composing
 three quasihomomorphisms: the injection of $C_i = A_0^{\sigma^{i-1}}$ into the $i$th coordinate of $A$; the quasiendomorphism $B$ from $A$ to $A^\sigma$; and the projection from $A^\sigma$ to its $j$th coordinate $D_j = A_0^{\sigma^j}$. Unlike the whole $B$, some of these $B_{ij}$ might be zero quasihomomorphisms, just as an invertible matrix might have zero entries.

 \item Morally, the original $B$ can be reconstituted from the matrix $\{B_{ij}\}$ as in the usual linear algebra, via the equation
 $$\hat{B} (x_1, x_2, \ldots x_m) = ( \sum_i B_{1i} x_i, \sum_i B_{2i} x_i, \ldots \sum_i B_{mi} x_i ),$$
 where $B_{ij} x_j$ means the quasihomomorphism action of $B_{ij}$ on the element $x_j$, and sums are in the sense of the group law of the appropriate $\sigma$-transform of $A_0$.
 More precisely, $\hat{B}$ is the appropriate fiber product of the $B_{ij}$; when $B_{ij}$ are not single-valued, $\hat{B}$ might not be equal to $B$. However, they always have the same connected component, so that the groups
 $G = (A, B)^{sh}$ and $(A, \hat{B})^{sh}$ always share a finite-index subgroup, and so, in particular, they have the same rank. Thus, we may and do work with $\hat{B}$ instead of $B$, which is to say that we work with the matrix $\{B_{ij}\}$.

 \item It follows immediately from the definition of $B$ that above the main diagonal, $B_{i, i+1}$ is the identity automorphism of $A_0^{\sigma^i}$; and that elsewhere except for the last row, $B_{i, j} = 0$ for any $i \neq m$ and any $j \neq i+1$. That is, $\{B_{ij}\}$ has the shape of a \emph{companion matrix}.

 \item The remaining entries $B_{mi}$ (in the last row of the matrix) are quasihomomorphisms from $A_0^{\sigma^{i-1}}$ to $A_0^{\sigma^m}$. If these two algebraic groups $A_0^{\sigma^{i-1}}$ and $A_0^{\sigma^m}$ are not isogenous, $B_{mi}$ must be the zero quasihomomorphism. In particular, $A_0$ and $A_0^{\sigma^m}$ must be isogenous, since otherwise $B_{i1} = 0$ for all $i$ and the matrix of $B_{ij}$s does not have full rank, which contradicts the choice of $m$ once you unwrap the construction back to that point.
 \end{itemize}
\end{fact}

Every simple commutative algebraic group $A_0$ is the additive group of the field, or the multiplicative group $\gm$ of the field, or a simple abelian variety. In characteristic zero, the additive group does not support any one-based groups, as its only endomorphisms are linear maps. If $A_0 = \gm$, then all  $B_{ij}$ come from $\mathbb{Q}$, the quasiendomorphism ring of $\gm$, so compositional divisibility of $B$ in the ring of quasiendomorphisms of $\gm^n$ becomes multiplicative divisibility of a matrix in the ring of $n \times n$ matrices over $\mathbb{Q}$. If $A_0$ is a simple abeliean variety defined over the field of absolute constants $F_\cap$, the same thing happens with $\mathbb{Q}$ replaced by the ring $R$ of quasiendomorphisms of $A_0$. For elliptic curves in characteristic zero, $R$ is a number field; otherwise, it may be more complicated. It could happen that the abelian variety $A_0$ is not defined over $F_\cap$, but is nevertheless isogenous to $\sigma_q(A_0)$ for many $q \in \mathbb{Q}$; this could probably happen even if $A_0$ is not fixed by any $\sigma_q$. Of all these interesting and maybe tractable possibilities, we only settle the special case where $A_0$ is the multiplicative group, or an elliptic curve defined over the absolute field of constants $F_\cap$, in characteristic zero.


\begin{notass} \label{gpquestsetup}
For the rest of the paper, we work in characteristic $0$, and\begin{itemize}
 \item $A_0$ is the multiplicative group, or an elliptic curve defined over the absolute field of constants $F_\cap$;
 \item $R$ is the quasiendomorphism ring of $A_0$;
 \item $m \geq 2$ and $M \in \operatorname{GLN}_m(R)$ is the companion matrix of its characteristic polynomial $P(x)$ in $R[x]$: that is, $M_{i,i+1} = 1$ for all $i$, and all other entries except for the last row of $M$ are zero.
 \item $G := \{ g \in A_0 : (\sigma(g), \sigma^2(g), \ldots \sigma^{m}(g)) = M \ast (g, \sigma(g), \ldots \sigma^{m-1}(g)) \},$ \\
     where $\ast$ is the quasiendomorphism action of $\operatorname{GLN}_m(R)$ on $A_0^{\times m}$, the $m$th cartesian power of $A_0$;
 \item Recall from ref{notation} that $\tau^n = \sigma$ and $a^+ = (a, \tau(a), \ldots \tau^{n-1}(a))$.
\end{itemize}
\end{notass}

\begin{fact} \label{dist-conj-eival} \begin{enumerate}
 \item All quasiendomorphisms of $A_0$ are defined over $F_\cap$.
 \item The ring $R$ is a number field.
 \item If $G$ is $L_\sigma$-one-based, then roots of unity are not eigenvalues of $M$.
 \item Subspaces of $R^m$ that are invariant under $M$ correspond to $L_\sigma$-definable subgroups of $G$, up to finite-index subgroups.
 \item In particular, if $G$ is $L_\sigma$-minimal and $L_\sigma$-one-based, then $M$ has no proper nontrivial invariant subspaces defined over $R$, its $m$ distinct eigenvalues in $R^{alg}$ form a conjugacy class over $R$, and the characteristic polynomial $P(x)$ of $M$ is irreducible over $R$.
\end{enumerate} \end{fact}

\begin{proof} \begin{enumerate}
 \item $A_0$ is defined over the algebraically closed field of absolute constants and has no algebraic families of quasiendomorphisms.
 \item This follows from the well-known characterization of endomorphism rings of elliptic curves.
 \item Otherwise $1$ would be an eigenvalue of $M^k$ for some $k$ and $G$ would be nonorthogonal to the fixed field of $\sigma^{m+k}$.
 \item The linear equations defining the subspace are precisely the quasiendomorphism equations defining the subgroups.
 Applying Facts \ref{whoisb0}, \ref{whoareABG}, and \ref{bijs} to an arbitrary $L_\sigma$-definable subgroup of $A_0$ shows that it must be defined by such quasiendomorphism equations. The invariance of the subspace under $M$ is equivalent to the minimality of $m$ in Fact \ref{whoisb0} for a subgroup of $G$.
 \item This is basic linear algebra.
\end{enumerate}\end{proof}

\subsection{More cases of Conjecture \ref{grouconj}.} \label{lastsec}

We now work out some linear algebra and algebraic number theory details towards proving Conjecture \ref{grouconj} in the setting of Notation \ref{gpquestsetup}.

\begin{lem} \label{gprolonglem} Up to finite-index subgroups, $G$ is also defined by
 $$h \in A_0\mbox{ and }
  (\tau(h), \tau^2(h), \ldots \tau^{mn}(h)) = \check{M} \ast (h,  \tau(h), \ldots \tau^{mn-1}(h))$$
 where $\check{M} \in \ggll_{mn}(R)$ is the companion matrix of the polynomial $P(x^n)$.
\end{lem}

\begin{proof} Fix $h \in A_0$; we must show that $h \in G$ if and only if it satisfies
 $$(\tau(h), \tau^2(h), \ldots \tau^{mn}(h)) = \check{M} \ast (h,  \tau(h), \ldots \tau^{mn-1}(h)).$$
 All but the last coordinates of the two sides of this equation are equal for any $h \in A_0$ and any companion matrix $\check{M}$. In the last row of the companion matrix $\check{M}$ of $P(x^n)$, the $(mn, (j-1)n+1)$th entry is $M_{mj}$ for each $j$, and the rest of the entries are zero.
  So the last coordinates of the two sides of this equation are equal if and only if
  $$\tau^{mn}(h) = \sum_j M_{mj} \tau^{(j-1)n}(h).$$
  Since $\tau^n=\sigma$, this is equivalent to
  $$\sigma^m(h) = \sum_j M_{mj} \sigma^{(j-1)}(h),$$
  which is the definition of $G$.
\end{proof}

\begin{lem} Bookkeeping Lemma. \label{bukip}\\
 Fix an integer $r$ and a matrix $L\in \operatorname{GLN}_r(R)$, and suppose that the group
$$H := \{h \in A_0:
  (\tau(h), \tau^2(h), \ldots \tau^{r}(h)) = L \ast (h,  \tau(h), \ldots \tau^{r-1}(h))$$
is (up to subgroups of finite index) a subgroup of $G$.
If $e \in R^{alg}$ is an eigenvalue of $L$, then $e^n$ is an eigenvalue of $M$.
\end{lem}

\begin{proof}
 Working in $L_\tau$, fix a generic realization $a$ of $H$. Let
 $$\Box(a) := (\tau^j \sigma^i (a))_{0 \leq i < m, 0\leq j <r} =$$
 $$= \begin{pmatrix}
 a & \tau(a) & \ldots & \tau^{r-1}(a) \\
 \sigma a & \tau(\sigma a) & \ldots & \tau^{r-1}(\sigma a) \\
 \vdots & \vdots & \ddots & \vdots \\
  \sigma^{m-1} a & \tau(\sigma^{m-1} a) & \ldots & \tau^{r-1}(\sigma^{m-1} a) \\
\end{pmatrix}.$$
 Since $\sigma$ and $\tau$ are automorphisms fixing $F_\cap$ over which $G$ and $H$ are defined, every entry in $\Box(a)$ also belongs to $H$ and $G$. We now compute $\sigma(\Box(a))$ in two different ways.

 On one hand, since $\sigma^i(a) \in H$, applying $\sigma=\tau^n$ to each row of $\Box(a)$ is the same as acting on it by $L^n$. So $\sigma(\Box(a))$ comes from applying $L^n$ to each row of $\Box(a)$; that is, acting by a block-diagonal matrix $\hat{L}$ with $m$ blocks, each of which is $L^n$.

 On the other hand, since $\tau^j(a) \in G$, applying $\sigma$ to each column of $\Box(a)$ is the same as acting on it by $M$. So $\sigma(\Box(a))$ comes from applying $M$ to each column of $\Box(a)$; that is, acting by a block-diagonal matrix $\tilde{M}$ with $r$ blocks, each of which is $M$.

 To match the ordering of the inputs, conjugate $\tilde{M}$ by a permutation matrix to get $\hat{M}$.
 Now $\hat{L} \Box(a) = \sigma(\Box(a)) = \hat{M} \Box(a)$ for any $a \in H$. It does not follow that $\hat{L} = \hat{M}$: for example, if $r > n$, many entries in $\Box(a)$ are repeated!

 However, $Z := \{ v \in A_0^{rm}: \hat{L} v = \hat{M} v\}$ is an algebraic subgroup of $A_0^{rm}$ that contains $\Box(a)$. Let $X \leq Z$ be the smallest algebraic subgroup $X$ of $A_0^{rm}$ containing $\Box(a)$.
 Since $a$ is generic in $H$, the first row of $\Box(a)$ is algebraically independent; so $X$ (and, therefore, $Z$) surjects onto the first $r$ coordinates. Indeed, $X$ has dimension $r$: it is defined precisely by ``the $i$th row is $L^{(i-1)n}$ applied to the first row.'' Now $X \leq Z$ means that for any $w \in A_0^r$, the quasihomomorphisms $\hat{L} \ast $ and $\hat{M} \ast$ agree on $(w, L^n(w), \ldots L^{(m-1)n}(w))$.

 Now, forgetting all about the algebraic group and quasiendomorphisms, we obtain a statement about linear algebra over $R$: for any $b \in R^r$, $\hat{L}$ and $\hat{M}$ agree on $\zeta(b)$, where
  $$\zeta : R^r \rightarrow R^{rm} : b \mapsto (b, L^n(b), \ldots L^{(m-1)n}(b))$$
 is the linear map represented by the block-diagonal matrix
 $\Lambda$ with blocks $L^{in}$ for $i = 0, 1, \ldots, m-1$.
 Equivalently, the matrices $\hat{L} \Lambda$ and $\hat{M} \Lambda$ are equal. Therefore, they also represent the same map from $(R^{alg})^r$ to $(R^{alg})^{rm}$.

 Let $e \in R^{alg}$ be an eigenvalue of $L$ with eigenvector $\vec{v} \in (R^{alg})^r$. Then $L^n (L^x(\vec{v})) = e^n L^x(\vec{v})$ for any $x$, so $\Lambda(\vec{v}) \in (R^{alg})^{rm}$ is an eigenvector of $\hat{L}$ with eigenvalue $e^n$. Now $(\hat{M} \Lambda) \vec{v} = (\hat{L} \Lambda) \vec{v} = e^n \Lambda \vec{v}$, so $\Lambda \vec{v}$ is also an eigenvector of $\hat{M}$ with eigenvalue $e^n$.
 Thus, $e^n$ is also an eigenvalue of $\tilde{M}$ which is conjugate to $\hat{M}$.
 Finally, since $\tilde{M}$ is block-diagonal with all blocks $M$, $e^n$ must also be an eigenvalue of $M$.
\end{proof}

 There ought to be a way to cut out all the matrices and eigenvalues and just talk about their characteristic polynomials and their roots; but this is the cleanest, convincingest proof I can write down. The next tool we use to analyze characteristic polynomials seems like a simple exercise in algebraic number theory, but we have not been able to find it in the literature.

\begin{deff}
 A \emph{hereditary factor} of a polynomial $P \in F[x]$ over a field $F$ is any factor of $P(x^n)$ for some $n \in \mathbb{N}$.

 A polynomial $P \in F[x]$ over a field $F$ is \emph{hereditarily irreducible over $F$} if for every $n \in \mathbb{N}$, the polynomial $P(x^n)$ is irreducible over $F$.
\end{deff}

\begin{lem} \label{algnumthlem}
 Let $F$ be a number field, and let $P \in F[x]$ be a polynomial none of whose roots in $F^{alg}$ are roots of unity. Then $P$ hereditarily factors into hereditarily irreducible polynomials: that is, for some $N \in \mathbb{N}$, the polynomial $P(x^N)$ is a product of polynomials over $F$ that are hereditarily irreducible over $F$.
\end{lem}

\begin{proof}
 Without loss of generality, $P$ itself is irreducible over $F$. Then $P$ has no multiple roots, and then neither does $P(x^n)$ for any $n$. The irreducible factors of $P(x^{n!}$ over $F$ form the $n$th level of a tree $T_0$, with a factor $R(x)$ of $P(x^{(n+1)!}$ lying above a factor $Q(x)$ of $P(x^{n!})$ whenever $R(x)$ divides $Q(x^{n+1}$.

 Now a factor $Q(x)$ of $P(x^{n!})$ is hereditarily irreducible if and only if there are no splits above it in this tree $T_0$. Let us trim $T_0$ by removing all nodes above such hereditarily irreducible factors, and call the new tree $T$. We need to show that $T$ is finite; suppose, towards a contradiction, that it is not. This infinite finitely-branching tree $T$ must have an infinite chain: integers $1= n_0 < n_1 < n_2 < n_3 < \ldots $ (all factorials, with $n_i$ dividing $n_{i+1}$ for each $i$) and irreducible factors $Q_i(x) \in F[x]$ of $P(x^{n_i}$ such that $Q_{i+1}(x)$ properly divides $Q_i(x^{\frac{n_{i+1}}{n_i}})$ for each $i$.

 Since $P$ is irreducible, any two roots of $P$ in $F^{alg}$ are conjugate by an automorphism $\rho$ of $F^{alg}$ over $F$. Since $\rho$ fixes the coefficients of $Q_i$, each root of $P$ has the same number $k_i \leq n_i$ of $n_i$th roots that are also of roots of $Q_i$. Since $Q_{i+1}(x)$ properly divides $Q_i(x^{\frac{n_{i+1}}{n_i}})$ for each $i$, we must have $\frac{k_{i+1}}{n_{i+1}} \lneq \frac{k_i}{n_i}$ for each $i$.

 Fix a root $a \in F^{alg}$ of $P$; now $\tilde{F} := F(a)$ is another number field, and $\tilde{P}(x) := x-a$ is an irreducible polynomial in $\tilde{F}[x]$. The $k_i$ $n_i$th roots of $a$ which are also roots of $Q_i$ form a Galois orbit over $\tilde{F}$ and thus correspond to a factor $\tilde{Q}_i \in \tilde{F}[x]$ of $Q(i)$ and also of $\tilde{P}(x^{n_i})$.

 The constant coefficient $a_i$ of $\tilde{Q}_i$ is (up to $\pm1$) the product of all the roots of $\tilde{Q}_i$, so it is the product of $k_i$ $n_i$th roots of $a$. That is, $a_i^{n_i} = \pm a^{k_i}$. That is, $a^{\frac{k_i}{n_i}} \in \tilde{F}$ for all $i$. This, together with the facts that $a$ is not a root of unity and $\frac{k_{i+1}}{n_{i+1}} \lneq \frac{k_i}{n_i}$ for each $i$ contradicts Unique Factorization (of ideals in the ring of integers of $\tilde{F}$) and/or Dirichlet's Theorem (that the group of units of the ring of integers of a number field is finitely generated).
\end{proof}

\begin{lem} \label{herirrcharlem}
 (Using Notation \ref{gpquestsetup}.) If $G$ is $L_\sigma$-one-based and the characteristic polynomial $P$ of $M$ is hereditarily irreducible over $R$, then $G$ is $L_\qq$-minimal. \end{lem}
\begin{proof}
 Since $\{ L_\tau : \tau^n =\sigma, n \in \mathbb{N} \}$ is a sufficient collection of subsignatures of $L_\qq$, by Corollary \ref{rankcomputesheep}, it suffices to show that $G$ is $L_\tau$-minimal for each such $\tau$.
 Suppose toward contradiction that this fails for some particular $n$. Since by ref{zilberesq} $G$ remains one-based in $L_\tau$, this failure must be witnessed by an $L_\tau$-definable, $L_\tau$-minimal subgroup $H$ of $G$, of infinite index.

 Analyzing $H$ as in Section \ref{group-redct-sect} yields an integer $r \lneq nm$ and a matrix $L \in \ggll_r(R)$ satisfying the hypotheses of Lemma \ref{bukip}. Let $Q(x)$ be the characteristic polynomial of $L$. By Lemma \ref{bukip}, all roots of $Q$ are $n$th roots of roots of $P$. By Lemma \ref{dist-conj-eival} applied to $H$, the polynomial $Q$ is irreducible over $R$, so it has no repeated roots. Thus, $Q(x)$ divides $P(x^n)$. Since the degree $r$ of $Q$ is strictly less than the degree $mn$ of $P(x^n)$, this contradicts hereditary irreducibility of $P$.
\end{proof}

\begin{prop} \label{happygpprop}  (Using Notation \ref{gpquestsetup}.) If $G$ is $L_\sigma$-one-based, then the $L_\qq$ Lascar rank of $G$ is finite and equal to the number of hereditarily irreducible hereditary factors of the characteristic polynomial of $M$. \end{prop}

\begin{proof}
Let $P \in R[x]$ be the characteristic polynomial of $M$. By Lemma \ref{algnumthlem}, $P(x^n)$ factors into hereditarily irreducible (over $R$) factors $Q_i(x)$ for some $n \in \mathbb{N}$.
Let $\tau$ be the named automorphism for which $\tau^n =\sigma$; by Lemma \ref{gprolonglem}, $G$ is the set of points
$h \in A_0$ such that
  $$(\tau(h), \tau^2(h), \ldots \tau^{mn}(h)) = \check{M} \ast (h,  \tau(h), \ldots \tau^{mn-1}(h))$$
for the companion matrix $\check{M}$ of $P(x^n)$.
Each irreducible factor $Q_i(x)$ of $P(x^n)$ corresponds to a subspace $V_i$ of $R^{mn}$ invariant under $\check{M}$, and $R^{mn}$ is the direct sum of these $V_i$.
 These subspaces in turn correspond to $L_\tau$=definable subgroups $H_i$ of $G$.
 Because the quasiendomorphism action is only defined up to finite noise, $G$ need not be the direct sum of $H_i$; but the natural homomorphism from $\sum_i H_i$ to $G$ is surjective (up to finite-index subgroups) onto $G$ and has a finite kernel. This is good enough to compute ranks: the rank of $\hat{G}$ will be the sum of ranks of $H_i$, not only in $L_\tau$, but also in any expansion.
Since the characteristic polynomials $Q_i$ of the matrices encoding $H_i$ are hereditarily irreducible, $H_i$ are $L_\qq$-minimal by Lemma \ref{herirrcharlem}. Thus the $L_\qq$-Lascar rank of $G$ is the number of these $H_i$s.
\end{proof}


\section*{Appendix}

Here we obtain the combinatorial characterization of rosiness in terms of dividing ranks stated in Fact \ref{adler-rosy-fact} from a theorem in \cite{adlerth} by unwrapping many definitions. All numbered references below are to \cite{adlerth}. We use notation from \cite{adlerth} without defining it.

\begin{fact} ( Theorem 2.37(3)) A theory $T$ is rosy if and only if $D_{\phi, \psi}(\emptyset) < \infty$ for every $(\phi, \psi) \in \Xi_M$. \end{fact}

 With Definition 2.23, this says that $T$ is not rosy if and only if there is some pair $(\phi, \psi) \in \Xi_M$ such that $\emptyset$ has $(\phi, \psi)$ dividing patterns of order-type $\omega$.
 At the end of this section we note that the empty type in Theorem 2.37(3) might as well be a partial type over the empty set; for the time being we do not make this assumption.

 \begin{lem} Unwrapping more definitions, Theorem 2.37(3) says that $T$ is not rosy if and only if there is some formula $\phi(\bar{x}; u \bar{v})$ and some $k$ such that $$( \wedge_{i < k} \phi(\bar{x}; u_i \bar{v}) ) \wedge ( \wedge_{i\neq j <k} u_i \neq u_j )$$ is inconsistent and $\emptyset$ has a $(\phi, \psi)$ dividing pattern of order-type $\omega$ for $\psi := (\wedge_{i\neq j <k} u_i \neq u_j) \wedge (\wedge_{i,j} \bar{v}_i = \bar{v}_j)$. \end{lem}

\begin{proof}

 Definition 2.35: $\Xi_M$ is the set of those pairs $( \phi(\bar{x}; u \bar{v}), \psi( (u \bar{v})_{<k}))$ in $\Xi$ where  $\psi$ is (equivalent to) ``all $u_i$ different, all $\bar{v}_i$ same''.
 With Definition 2.9, $\Xi_M$ is the set of those pairs $( \phi(\bar{x}; u \bar{v}), \psi( (u \bar{v})_{<k}))$ where  $\psi$ says ``all $u_i$ different, all $\bar{v}_i$ same'' and is a $k$-inconsistency witness for $\phi$.

 With more Definition 2.9, $\Xi_M$ is the set of those pairs $( \phi(\bar{x}; u \bar{v}), \psi( (u \bar{v})_{<k}))$ where $\psi := (\wedge_{i\neq j <k} u_i \neq u_j) \wedge (\wedge_{i,j} \bar{v}_i = \bar{v}_j)$ and
  $$ \Lambda := ( \wedge_{i < k} \phi(\bar{x}; u_i \bar{v}_i) ) \wedge  \psi$$
 is inconsistent. Note that $\Lambda$ is inconsistent if and only if $\Lambda' := ( \wedge_{i < k} \phi(\bar{x}; u_i \bar{v}) ) \wedge ( \wedge_{i\neq j <k} u_i \neq u_j )$ is inconsistent.
 Also, as noted after Definition 2.35, $\Lambda'$ is inconsistent if and only of whenever $\phi(\bar{a}, b \bar{c})$ holds, $b$ is algebraic over $\bar{a}\bar{c}$.
\end{proof}

 Now let us unwrap the definitions of dividing patterns, keeping in mind the identity of $\psi$ and the separation of the second set of variables.

\begin{lem} The empty type over $C$ has a $(\phi, \psi)$ dividing pattern of order-type $\omega$ for $\psi := (\wedge_{i\neq j <k} u_i \neq u_j) \wedge (\wedge_{i,j} \bar{v}_i = \bar{v}_j)$ if and only if there are $(b^i, \bar{c}^i, \bar{d}^i)_{i \in \omega}$ and $b^{ij}$ for $j \in \omega$ such that
\begin{itemize}
\item $\{ \phi(x, b^i, \bar{c}^i) \sthat i \in \omega \}$ is consistent;
\item $b^{ij}, \bar{d}^i \equiv_{A_i} b^{i}, \bar{c}^{i}$ for $A_i := C \cup \{ b^i, \bar{c}^i \sthat j < i \}$; and
\item $b^{ij} \neq  b^{ij'}$ for all $j \neq j'$.
\end{itemize} \end{lem}
\begin{proof}

 Definition 2.20, special case of $\Delta = \{ (\phi(\bar{x}; u \bar{v}), \psi) \}$,
  $I = \omega$, $p = \emptyset$ is a partial type over $C$:
 ``$\emptyset$ has $(\phi, \psi)$ dividing patterns of order-type $\omega$'' means that there are $(b^i, \bar{c}^i)_{i \in \omega}$ such that $\{ \phi(x, b^i, \bar{c}^i) \sthat i \in \omega \}$ is consistent, and $\phi(x, b^i, \bar{c}^i)$ $(\phi, \psi)$-divides over $A_i := C \cup \{ b^i, \bar{c}^i \sthat j < i \}$.

 Definition 2.10: ``$\phi(x, b^i, \bar{c}^i)$ $(\phi, \psi)$-divides over $A_i$'' means that there are $b^{ij}, \bar{c}^{ij}$ for $j \in \omega$ such that $b^{ij}, \bar{c}^{ij} \equiv_{A_i} b^{i}, \bar{c}^{i}$ and $\psi$ holds on any $k$-tuple of $b^{ij}, \bar{c}^{ij}$ with increasing $j$s.

 Considering who $\psi$ is, this says that $b^{ij} \neq  b^{ij'}$ for all $j \neq j'$ and
  $\bar{c}^{ij} = \bar{c}^{ij'}$ for all $j, j'$, so we may replace $\bar{c}^{ij}$ by $\bar{d}^i$.
\end{proof}

With these two lemmas,  $T$ is not rosy if and only if there are a formula $\phi(\bar{x}; u \bar{v})$ and an integer $k$ such that
 $$( \wedge_{i < k} \phi(\bar{x}; u_i \bar{v}) ) \wedge ( \wedge_{i\neq j <k} u_i \neq u_j )$$
 is inconsistent, and there are a set $C$ and parameters $b^i, \bar{c}^i, \bar{d}^i$ and $b_{ij}$ for $i, j \in \omega$ such that:
\begin{itemize}
\item $\{ \phi(x, b^i, \bar{c}^i) \sthat i \in \omega \}$ is consistent;
\item $b^{ij}, \bar{d}^i \equiv_{A_i} b^{i}, \bar{c}^{i}$ for $A_i := C \cup \{ b^i, \bar{c}^i \sthat j < i \}$; and
\item $b^{ij} \neq  b^{ij'}$ for all $j \neq j'$.
\end{itemize}

 The characterization in Fact \ref{adler-rosy-fact} follows from noting that if any set $C$ witnesses this, then so does $C = \emptyset$, as promised at the beginning of this section.

\begin{prop}
A theory $T$ is not rosy if and only if there are:
 a formula $\phi(\bar{x}; u \bar{v})$, an integer $k$, a model $M \models T$, and parameters $b^i, \bar{c}^i, \bar{d}^i$ and $b_{ij}$ for $i, j \in \omega$ in $M^{eq}$ such that \begin{itemize}
 \item $( \wedge_{i < k} \phi(\bar{x}; u_i \bar{v}) ) \wedge ( \wedge_{i\neq j <k} u_i \neq u_j )$
 is inconsistent;
 \item $\{ \phi(x, b^i, \bar{c}^i) \sthat i \in \omega \}$ is consistent;
\item $b^{ij}, \bar{d}^i \equiv_{A_i} b^{i}, \bar{c}^{i}$ for $A_i := \{ b^i, \bar{c}^i \sthat j < i \}$; and
\item $b^{ij} \neq  b^{ij'}$ for all $j \neq j'$.
\end{itemize}
\end{prop}

\bibliographystyle{plain}
\bibliography{qacfa}

\end{document}